\documentclass[a4paper]{article}
\usepackage{a4wide}
\usepackage{latexsym}
\usepackage[UKenglish]{babel}
\usepackage{graphicx,subfigure}
\usepackage{algorithm}
\usepackage{algpseudocode}
\usepackage{epstopdf}
	\graphicspath{{./}{figures/}}
\usepackage[colorlinks=true,linkcolor=blue,citecolor=blue]{hyperref}
\usepackage{xcolor}
\usepackage{amsfonts,amsmath,amsthm,mathtools,bbm,cool}
\newcommand{\be}{\begin{equation}}
\newcommand{\ee}{\end{equation}}
\newenvironment{equations}{\equation\aligned}{\endaligned\endequation}
\newcommand{\fer}[1]{(\ref{#1})}
\def\ff{\widehat f}

\def\gg{\widehat g}
\newcommand{\R}{\mathbb{R}}

\newtheorem{theorem}{Theorem}
\newtheorem{lemma}{Lemma}
\newtheorem{remark}{Remark}
\newtheorem{definition}{Definition}

\makeatletter
\def\blfootnote{\xdef\@thefnmark{}\@footnotetext}
\makeatother

\begin{document}
\title{Large-time behaviour for coupled systems of Lotka-Volterra-type Fokker-Planck equations}

\author{ Giuseppe Toscani \\
		{\small	Department of Mathematics ``F. Casorati''} \\
		{\small University of Pavia, Italy} \\
		{\small\tt giuseppe.toscani@unipv.it} \\
		 Mattia Zanella \\
		{\small	Department of Mathematics ``F. Casorati''} \\
		{\small University of Pavia, Italy} \\
		{\small\tt mattia.zanella@unipv.it} 
		}

\maketitle

\vspace*{-0.8cm}
\begin{abstract}
We study a system of Fokker-Planck equations recently introduced to describe the temporal evolution of statistical distributions of population densities with predator-prey interactions. At the macroscopic level, the system recovers a Lotka-Volterra model and defines an explicit family of equilibrium densities that depend on the form of the diffusion coefficient. By introducing Energy-type distances, we rigorously establish exponential convergence to equilibrium in appropriate homogeneous Sobolev spaces, with a rate explicitly determined by the dissipative contribution of the interaction term. The analysis highlights the intrinsic energy dissipation mechanism governing the dynamics and clarifies how the evolution of expected quantities determines the emergence of a stable equilibrium configuration. This approach provides a new perspective on the convergence to equilibrium for problems with time-dependent coefficients. 
\end{abstract}

\vspace*{0.5cm}

\vspace*{0.5cm}

\tableofcontents

\section{Introduction}
In this paper, we study convergence to equilibrium of the solution of a system of Fokker--Planck type equations which describes the evolution of the probability densities ${\mathbf f}(x,t) = (f_1(x,t), f_2(x,t))$, $x \in \R_+$,  of an interacting particle system characterized by two groups of agents/particles. {In the present setting $x \in \mathbb{R}_+$ denotes the population density, whose evolution is interpreted in a statistical sense \cite{BMTZ,MTZ}.} In the following, this system will feature predator-prey interactions and such that their mean values ${\mathbf m}(t)= (m_1(t), m_2(t))$, given by
 \be\label{intro:mean}
 m_k(t) = \int_{\R_+} x\, f_k(x,t)\, dx, \qquad k =1,2,
 \ee
 follow a Lotka-Volterra dynamics
 \begin{equation}
\label{intro:LV}
\begin{split}
\dfrac{dm_1}{dt} &= \alpha\left( 1-\dfrac{m_1}{K} \right)m_1  - \beta m_1m_2 \\
\dfrac{dm_2}{dt} &= -\delta m_2 + \gamma m_1 m_2, 
\end{split}
\end{equation} 
{where $\alpha>0$ is the growth rate, $K>0$ the carrying capacity, $\beta>0$ the removal rate due to contact between agents in the first population interacting with agents in the second population, $\delta>0$ a natural decay rate in the second specie, and $\gamma>0$ the growth rate of the second population due to interactions between populations. The system \eqref{intro:LV} corresponds to the classical Lotka-Volterra model with logistic growth \cite{Pearl}}. At the macroscopic level, these interactions shape the mutual evolution of this coupled system of particles. A rich literature  concentrating on species interacting through space-dependent factors have been developed in recent years, see e.g. \cite{AFMS,BE,CHS18,ChaJabRao,chen21,CDJ,ChoLuiYam_Strong,DIEKMANN2005257,HHMM,Yam}. In the direction of interacting systems with cross-diffusion we mention the works \cite{ARSW,BHS,Canizo,CD,RBM_DAUS} and \cite{Jabin_2017,Jabin_11,Pouchol01012018} for applications to structured populations. A kinetic approach to such-dynamics have been proposed in \cite{BMTZ,TosZan} where a multi-scale perspective connecting kinetic theory with classical models of population dynamics have been proposed. In particular, in these works a system of mean-field equations have been derived from a microscopic dynamics through the approaches defined in \cite{APZ,DPTZ,ParTos-2013}.  The  system of Fokker--Planck type equations derived in \cite{BMTZ} reads
\begin{equation}
\label{intro:systFP}
\begin{split}
\dfrac{\partial}{\partial t} f_1(x,t) =& \dfrac{\partial}{\partial x} J_1(f_1(t),t)= \dfrac{\sigma_1^2(t)}{2} \dfrac{\partial^2}{\partial x^2} (x^{2p} f_1(x,t)) + 
 \dfrac{\partial}{\partial x} \left[ \left( \lambda_1(t)x  - \mu_1(t)\right) f_1(x,t)\right] \\
\dfrac{\partial}{\partial t} f_2(x,t) =& \dfrac{\partial}{\partial x} J_2(f_2(t),t)= \dfrac{\sigma_2^2(t)}{2} \dfrac{\partial^2}{\partial x^2} (x^{2p} f_2(x,t)) + 
 \dfrac{\partial}{\partial x} \left[ \left( \lambda_2(t)x  - \mu_2(t)\right) f_2(x,t)\right]. 
\end{split}
\end{equation}
In this system the constant $1/2 \le p \le 1$ measures the intensity of random effects, while the coefficients of diffusion and drift are functions of time through the mean values of the solution, as given by \fer{intro:mean}
\be\label{intro:coeff}
\sigma_k(t) = \sigma_k[{\mathbf m}(t)], \quad {\lambda_k}(t) = \lambda_k[{\mathbf m}(t)], \quad \mu_k(t) = \mu_k[{\mathbf m}(t)]; \quad k =1,2,
\ee
characteristics which renders the study of the large--time behaviour of the solutions a challenging problem since the evolution of the mean values is non-monotone in time, in contrast to classical approaches in kinetic theory \cite{Arnold,CerIllPul}. {We recall that the bound $1/2\le p\le 1$ is linked to the integrability of equilibrium densities. } System \eqref{intro:systFP} is supplemented with no-flux boundary conditions that guarantee the conservation of total mass.

Relaxation problems for mean-field equations with time-dependent coefficient have been also considered in \cite{ATZ,ACTZ}. 
Nevertheless, the analysis discussed in \cite{ATZ} has been made possible thanks to the explicit parametrisation of the quasi-equilibrium density which enabled the application of classical entropy methods, see e.g. \cite{MR3597010,FPTT19,Bris03072008} for related approaches. 
We recall that, given a Fokker--Planck equation of type \eqref{intro:systFP}, namely
\be\label{intro:qs}
\frac{\partial}{\partial t} f(x,t) = \frac{\partial}{\partial x} J(f(x,t),t),
\ee
the \emph{quasi-equilibrium distribution} is the time-dependent probability density 
\( f^q(x,t) \), \( x \in \mathbb{R}_+ \), satisfying the first-order equation 
\( J(f^q(x,t),t) = 0 \).
When considering system \eqref{intro:systFP}, it can be shown that the quasi-equilibria
\[
\mathbf{f}^q(x,t) = \big(f_1^q(x,t), f_2^q(x,t)\big), \qquad x \in \mathbb{R}_+,
\]
are given by generalized Gamma distributions~\cite{Sta}, whose form depends critically on 
the parameter \( p \) that characterizes the \( x \)-dependent diffusion coefficient.
In particular, the limiting cases \( p = \tfrac{1}{2} \) and \( p = 1 \) correspond, respectively,
to Gamma and inverse-Gamma quasi-equilibrium densities~\cite{BMTZ}. Systems of Fokker--Planck type equations with mean values satisfying  the classical Lotka--Volterra systems have been originally proposed in \cite{TosZan} to model the distributions of the populations sizes of both preys and predators in presence of random effects.  The system introduced in \cite{TosZan} has been subsequently studied in details in \cite{BMTZ}, where however the description of the large--time behavior of the solutions has been left open. The behavior of a similar system has been subsequently investigated in \cite{MT}  motivated by the study of the evolution of the inequality expressed by the probability densities for two populations interacting pairwise by economic motivations.  

The choice of the logistic case is motivated by the fact that, when 
\( { (\gamma K - \delta) > 0} \), the solution \( \mathbf{m}(t) \) converges to the unique equilibrium point
\begin{equation}
\label{intro:equilibrium_mean}
\mathbf{m}^\infty = \left( \frac{\delta}{\gamma}, \frac{\alpha(\gamma K - \delta)}{\beta \gamma K} \right).
\end{equation}
The convergence of the mean values towards \( \mathbf{m}^\infty \) implies the convergence 
of the time-dependent coefficients \eqref{intro:coeff} towards their limiting values 
\( \sigma_k(\mathbf{m}^\infty) \), \( \lambda_k(\mathbf{m}^\infty) \), and 
\( \mu_k(\mathbf{m}^\infty) \), \( k = 1,2 \). 
Consequently, the quasi-equilibria \( f_k^q(x,t) \), \( k = 1,2 \), converge towards the 
corresponding generalized Gamma densities, denoted by \( f_k^\infty(x) \), \( k = 1,2 \). 
In this setting, it is reasonable to conjecture that the solutions 
\( \mathbf{f}(x,t) \) converge, in a suitable sense, towards the equilibrium density 
\( \mathbf{f}^\infty(x) = (f_1^\infty(x), f_2^\infty(x)) \).

In a recent works~\cite{LTZ,MTZ}, the large-time behaviour of a system of Fokker--Planck type equations with time-dependent diffusion coefficients has been investigated by means of \emph{Energy distances}, a well-known class of measures of discrepancy between probability distributions, widely used in Statistics~\cite{SR} and Artificial Intelligence~\cite{Belle}. 
This analysis suggests that such (weak) distances are particularly well suited for studying the convergence properties of kinetic systems with time-varying parameters. Energy distances, introduced by Székely (cf.~\cite{Sze89, Sze03, SR}), can be regarded as a natural generalization of Cramér’s distance~\cite{Cra}, also known as the \emph{continuous ranked probability score}. 
They represent a valid alternative to other well-known metrics, such as the Wasserstein distance~\cite{villani2009optimal} and the Fourier-based metrics introduced in~\cite{GTW}, which have been extensively employed in kinetic theory (see~\cite{ABGT2} for a recent overview of their properties and mutual relationships).
A rigorous validation of the suitability of Energy distances for studying convergence to equilibrium in linear kinetic equations was presented in~\cite{AT}. 
There, it was shown that these distances provide an effective tool to quantify the rate of convergence to equilibrium for solutions of linear diffusion equations and Fokker--Planck type equations with linear drift, including the one-dimensional socio-economic models analysed in~\cite{MR3597010}. 
In all considered cases, the solutions exhibit exponential decay in time towards equilibrium.

In this paper, we show that Energy distances provide an effective quantitative framework for analysing the convergence to equilibrium of the solutions to systems of the type introduced in~\eqref{intro:systFP}. 
In particular, in the limiting cases \( p = \tfrac{1}{2} \), yielding equilibria of Gamma type, and \( p = 1 \), corresponding to inverse Gamma equilibria, the structure of the diffusion operator allows the use of Fourier-based techniques to obtain  decay estimates in a wide class of Energy distances. 
This approach leads to a rigorous characterisation of the rate of convergence towards equilibrium, highlighting the dissipative nature of the dynamics and the dependence of the asymptotic behaviour on the underlying parameters.

In more detail, in Section \ref{sect:2} we introduce the Energy distance and we highlight its connection with homogeneous Sobolev spaces of fractional order. In Section \ref{sect:3} we introduce the system of Fokker-Planck equations with time-dependent coefficients recalling the evolution of macroscopic quantities like mean and variance and we briefly study the convergence to equilibrium for the resulting systems of moment equations. In Section \ref{sect:4} we discuss the general problem of convergence to equilibrium of the solutions to the Fokker--Planck system in the full range of $p \in [\frac 1 2, 1]$, with a detailed study of the cases $p = 1/2$ and $p = 1$,   reckoning the rate of convergence to the equilibrium distributions.  Several numerical results will be presented in Section \ref{sect:5} to highlight the effectiveness of the theoretical results. 

\section{The Energy distance}\label{sect:2}
In this section we introduce the main tools that will be considered for the study of equilibration of the system of Fokker-Planck equations \eqref{intro:systFP}.  Let $f(x)\ge 0$, $x \in \R_+$ be a probability density, i.e. 
\[
\int_{\R_+} f(x) \,dx =1,
\]
and let $\hat f$ denote its Fourier transform 
\[
\hat f(\xi) = \int_{\mathbb R_+} f(x)e^{-i\xi x}dx. 
\]
We will associate $f$ to a random variable $X$, such that for a subset of the reals $A \subseteq \R$, we have 
\[
F(X\in A) = \int_A f(x)\, dx.
\]
The cumulative distribution function of $X$ is then
\be\label{cumu}
F(x) = P(X \le x) = \int_{0}^x f(x)\, dx.
\ee
Given $r>0$, we denote by $P_r(\R_+)$ the class of all probability densities $f$ over $\R_+$ such that
\[
m^{(r)}  = \int_{\R_+} x^r \,f(x)\, dx < + \infty.
\]
 If $r=1$, $ m^{(1)}= m$ defines the mean value of the random variable $X$ and the variance is defined as follows 
 \[
V = m^{(2)} - m^2. 
 \]
 
 \begin{definition}
Given two random variables $X$ and $Y$ of probability densities $f(x)$ and  $g(x)$ respectively, $x\in \R_+$, the Energy distance of order $r$, with $0<r<2$  between  $X$ and $Y$ is defined as follows 
\begin{equations}\label{energy-a}
\mathbf{E}_r (f,g) &= 2\int_{\R_+^2}|x-y|^rf(x)g(y) \, dx dy\\
& -\int_{\R_+^2}|x-y|^rf(x)f(y) \, dx dy
 -\int_{\R_+^2}|x-y|^r g(x)g(y) \, dx dy\\
&= - \int_{\R_+^2}|x-y|^r [f(x)-g(x)][f(y) -g(y)] \, dx dy.
\end{equations} 
\end{definition}

We observe that the introduced distance can be equivalently expressed in Fourier transform as follow 
\be\label{energy2}
\mathbf{E}_r (f,g) = c_{r} \int_{\R} \frac{ |\ff(\xi) -\gg(\xi)|^2}{|\xi|^{1+r}}\, d\xi,
\ee
where
\be\label{c1}
c_{r} =  \frac{r\,\Gamma\left( \frac{1+r}2\right)}{  2^{1-r}\sqrt\pi\, \Gamma\left( \frac{2-r}2\right)}, 
\ee
see \cite{ABGT2,SR}.  Thanks to the reformulation of this distance in terms of Fourier transformed densities we may easily verify that right-hand side in \fer{energy-a} is nonnegative, so that the mixed integral dominates the sum of the last two. Furthermore, expression \fer{energy2} connects the Energy distance with norms arising in the study of homogeneous Sobolev spaces of fractional order $\dot H_{-r}$, being
\be\label{hq}
\| f\|_{\dot H_{-r}} = \int_\R |\xi|^{-2r}|\hat f(\xi)|^2 \, d\xi,
\ee
see \cite{BG}. We recall that, if $r=1$, we have
\be\label{en-cra}
\mathbf{E}_1 (f,g) = \frac 1{\pi} \int_{\R} \frac{ |\ff(\xi) -\gg(\xi)|^2}{|\xi|^{2}}\, d\xi,
\ee
which corresponds to twice the Cram\'er distance, classically defined as the square of the $L^2$-norm of the difference between the cumulative functions of $X$ and $Y$, that is
\be\label{Cramer}
d(f,g) = \int_\R[F(x) -G(x)]^2\, dx.
\ee 
Indeed, for any given pair of probability densities $f,g \in P_1(\R)$,  Parseval formula implies
\be\label{Par}
\int_\R [F(x) - G(x)]^2 \, dx = \frac 1{2\pi} \int_\R |\widehat F(\xi) - \widehat G(\xi)|^2 \, d\xi,
\ee
where $\widehat F$ and $ \widehat G$ are the Fourier transforms of the distribution functions $F,G$. If
$\ff(\xi)$ and $ \gg(\xi)$ are the Fourier transforms of $f$ and $g$
 integration by parts gives 
 \[
 \widehat F(\xi) - \widehat G(\xi) = \frac{\ff(\xi) -\gg(\xi)}{i\xi}.
 \]
  Consequently, the Cram\'er distance defined in \fer{Cramer} can be equivalently written as
 \be\label{Cramer2}
d(F,G)  = \frac 1{2\pi} \int_\R \frac{|\ff(\xi) -\gg(\xi)|^2}{|\xi|^2} \, d\xi.
\ee
To avoid inessential computations, in the rest of the paper we will simply refer to a \emph{normalized} energy distance of order $\ell$,  $\mathcal E_\ell$, which is obtained by setting $c_r =1$ and $\ell=(1+r)/2$, with $1/2 <\ell<3/2$, i.e.
\be\label{eq:norm}
\mathcal{E}_\ell (f,g) =  \int_{\R} \frac{ |\ff(\xi) -\gg(\xi)|^2}{|\xi|^{2\ell}}\, d\xi.
\ee 
\begin{remark}
The normalized Energy distance defined in \fer{eq:norm}, while maintaining connection with the widely used expression \fer{energy2} introduced by Sz\'ekely in \cite{Sze89}, is such that its square root is nothing but the homogeneous Sobolev space of fractional order with negative index $\dot{H}_{-\ell}$. Thus, the square root of the normalized Cram\'er distance coincides with $\dot{H}_{-1}$.
\end{remark}
It is interesting to observe that, for any given $\ell<\ell^*$, the Energy distance of order $\ell^*$ controls the energy distance of order $\ell$. More precisely, the following quantitative result holds 
\begin{lemma}\label{lemma1}
Let $f(x),g(x)$ be two probability densities, $x \in \mathbb R_+$, and let $\ell,\ell^* \in \left(\frac12,\frac32\right)$ such that $\ell<\ell^*$. If $\mathcal{E}_\ell$ is the energy distance of order $\ell$ defined in \fer{eq:norm} it holds
\be\label{eq:scala}
\mathcal E_\ell(f,g)^{2\ell^*-1} \le C_{\ell,\ell^*} \mathcal E_{\ell^*}(f,g)^{2\ell-1},
\ee
where
\be\label{cab}
C_{\ell,\ell^*}= \left(\frac4{\ell^*-\ell}\right)^{2(\ell^*-\ell)} \left(\dfrac{2\ell^*-1}{2\ell-1} \right)^{2\ell-1} 
\ee
\end{lemma}

\begin{proof}
Consider that, for given $\ell<\ell^*$, and any $R>0$ the following inequalities hold
\begin{equations}
& \int_{\R} \frac{ |\ff(\xi) -\gg(\xi)|^2}{|\xi|^{2\ell}}\, d\xi = \int_{|\xi|\le R} \frac{ |\ff(\xi) -\gg(\xi)|^2}{|\xi|^{2\ell}}\, d\xi + \int_{|\xi|>R} \frac{ |\ff(\xi) -\gg(\xi)|^2}{|\xi|^{2\ell}}\, d\xi \\
&\le R^{2(\ell^*-\ell)} \int_{|\xi|\le R} \frac{ |\ff(\xi) -\gg(\xi)|^2}{|\xi|^{2\ell^*}}\, d\xi  + 4 \int_{|\xi|>R} \frac{ 1}{|\xi|^{2\ell}}\, d\xi
\end{equations}
since for any probability density $f(x,t)$ we have
\[
 |\hat f(\xi,t)| =  \left|\int_{\mathbb R_+}e^{-i\xi x} f(x,t)\right|dx\le\int_{\mathbb R_+} |f(x,t)|dx = 1. 
\]
Therefore, we get
\begin{equation}
\begin{split}
 &R^{2(\ell^*-\ell)} \int_{|\xi|\le R} \frac{ |\ff(\xi) -\gg(\xi)|^2}{|\xi|^{2\ell^*}}\, d\xi  + 4 \int_{|\xi|>R} \frac{ 1}{|\xi|^{2\ell}}\, d\xi \le \\
 &\qquad R^{2(\ell^*-\ell)} \int_{\R} \frac{ |\ff(\xi) -\gg(\xi)|^2}{|\xi|^{2\ell^*}}\, d\xi  + \frac 8{(2\ell-1) R^{2\ell-1}}, 
\end{split}\end{equation}
and setting 
\be\label{eq:opt}
h(R) = R^{2(\ell^*-\ell)} \int_{\R} \frac{ |\ff(\xi) -\gg(\xi)|^2}{|\xi|^{2\ell^*}}\, d\xi  + \frac 8{(2\ell-1) R^{2\ell-1}}, 
\ee
we find the optimal value
\[
R^* = \left(\dfrac{4}{(\ell^*-\ell) \displaystyle\int_{\mathbb R} \dfrac{|\hat f(\xi) - \hat g(\xi)|^2}{|\xi|^{2\ell^*}}d\xi} \right)^{1/(2\ell^*-1)},
\]
Substituting this value into \fer{eq:opt} gives \fer{eq:scala}.
\end{proof}

Consequently, once we know that a sequence of probability densities $f_n$, $n\ge 1$, converges to a density $f$ in the energy distance of order $\ell^*$, we can conclude that convergence of the sequence holds in the energy distance of order $\ell$ for any value $\ell<\ell^*$.

\section{Lotka-Volterra-type kinetic equations}\label{sect:3}

The system \fer{intro:systFP} of Fokker--Planck equations was introduced in \cite{TosZan} to describe, in a simplified situation, the time evolution of the size distribution of two populations governed by predator--prey interactions. In \cite{TosZan} (cf. also \cite{BMTZ}) the evolution in time of the population's sizes $f_1(x,t)$ and $f_2(x,t)$ has been shown to obey in general to a system of  Boltzmann-type equations, in which  the dynamics are provided by binary interactions between  predator and prey populations, while births and deaths are managed by a linear distribution operator, originally introduced in \cite{BisSpiTos} to model wealth taxation in a society of agents. In a suitable \emph{grazing} scaling regime,  the Boltzmann formulation is close to a system of coupled Fokker--Planck-type equations for the probability densities  representing the statistical distributions of preys, and, respectively, of predators.

 In detail, let $f_1(x,t), f_2(x,t)$, $(x,t) \in \mathbb R_+^2$,  denote the probability density functions of two types of agents/particles, typically referred to as preys and predators, respectively. Accordingly, {$f_k(x,t)dx, \,\, k=1,2$} quantifies the fraction  of group {$k=1,2$} with size in $[x,x+dx]$ at time $t\ge0$. Moreover, let $m_1(t),m_2(t)$, as given by \fer{intro:mean},  the mean values of the populations. 
 
 \subsection{Lotka-Volterra-type Fokker-Planck equations }

 The evolution in time of the mean values of the biological system admits a hierarchical finer description. Building on the ideas in \cite{BMTZ,TosZan}, one can define a mass-preserving non-Maxwellian system of kinetic equations for the probability densites of a mixture of particles such that, at the macroscopic level, the mean values obey to system \fer{intro:LV}, which corresponds to the classical Lotka-Volterra equation with logistic growth 
 \begin{equation}
\label{eq:LV}
\begin{split}
\dfrac{dm_1(t)}{dt} &= \alpha\left( 1-\dfrac{m_1(t)}{K} \right)m_1(t)  - \beta m_1(t)m_2(t) \\
\dfrac{dm_2(t)}{dt} &= -\delta m_2(t) + \gamma m_1(t) m_2(t). 
\end{split}
\end{equation} 
In \eqref{eq:LV} the coefficients $\alpha>0$ indicate the growth rate, $K>0$ a carrying capacity of the system, $\beta>0$ the removal rate due to contact between particles in the first population interacting with particles in the second population, $\delta>0$ a natural decay rate of the second specie, and $\gamma>0$ the growth rate of the second population due to interacting between the two populations.  As briefly described in the Introduction, the unique coexistence equilibrium point for the system of the mean values is 
\begin{equation}
\label{eq:equilibrium_mean}
\mathbf m^\infty = \left( \dfrac{\delta}{\gamma},\dfrac{\alpha(\gamma K - \delta)}{\beta \gamma K}\right), 
\end{equation}
which is meaningful under the hypothesis that {$\gamma K-\delta>0$}. At variance with a purely microscopic approach, the new system of kinetic equations is able to propagate information on the whole moment system of the kinetic distributions $f_1,f_2$ {provided by the quantities
\begin{equation}\label{eq:mom2}
m_k^{(\ell)} = \int_{\mathbb R_+}x^\ell f_k(x,t)dx, \qquad \ell>0, k=1,2 
\end{equation}}
Nevertheless, as investigated in \cite{BMTZ}, this systems become rapidly unfeasible to direct inspections since nonlinear terms appears starting from the evolution of the second order moment, {i.e. $\ell = 2$ in \eqref{eq:mom2}}. For those reasons, a reduced complexity system of equations have been proposed in \cite{BMTZ,TosZan} building on quasi-invariant limit methods for collisional kinetic equations, see \cite{ParTos-2013} for a comprehensive introduction. 

In a suitable grazing limit, from the bilinear Boltzmann-type description for the evolution of $f_1,f_2$  one obtains a reduced complexity system of Fokker-Planck equations, in which the detailed bilinear interactions are substituted by mean values. As shown in \cite{BMTZ} the reduced system is of type \fer{intro:systFP}, where
\begin{equations}\label{eq:values}
& \sigma_1^2(t) = \sigma_1^2(m_1(t) + m_2(t)) \quad \lambda_1(t) =   \beta m_2(t)  + \dfrac{\alpha}{K}m_1(t) + \alpha \chi, \quad \mu_1(t) =\alpha(\chi + 1)m_1(t),  \\
& \sigma_2^2(t) = \sigma_2^2m_1(t), \quad \lambda_2(t) = \gamma(\mu-m_1(t)) + \nu\theta, \quad \mu_2(t) = \nu(\theta+1)m_2(t).
\end{equations}
In \fer{eq:values} the positive constants  $\sigma_1^2, \sigma_2^2>0$ measure the strength of the diffusion, while the $\theta,\chi>-1$ and $\nu>0$ are related to the rate of birth and death of particles. {Furthermore, according to \cite{BMTZ} the parameter $\delta>0$ is such that $\delta = \gamma\mu-\nu$.} For large times we have
\begin{equations}\label{eq:values_large}
& \sigma_{1,\infty}^2= \sigma_1^2(m_1^\infty + m_2^\infty) \quad \lambda_{1,\infty} =   \beta m_2^\infty  + \dfrac{\alpha}{K}m_1^\infty + \alpha \chi, \quad \mu_1^\infty =\alpha(\chi + 1)m_1^\infty,  \\
& \sigma_{2,\infty}^2 = \sigma_2^2m_1^\infty, \quad \lambda_{2,\infty} = \gamma(\mu-m_1^\infty) + \nu\theta, \quad \mu_{2,\infty} = \nu(\theta+1)m_2^\infty.
\end{equations}

\subsection{Evolution of observable quantities for the mean-field model}
\label{subsect:varFP}

The principal moments of the probability densities can be fruitfully evaluated by resorting to the weak form of the coupled system of Fokker-Planck equations \eqref{intro:systFP}. The weak form corresponds to
\begin{equation}
\label{eq:systFP_weak}
\begin{split}
\dfrac{d}{dt} \int_{\mathbb R_+}\varphi(x)f_1(x,t)dx &= \dfrac{\sigma_1^2(t)}2 \int_{\mathbb R_+}\varphi^{\prime\prime}(x)\left(x^{2p}f_1(x,t)\right)dx  - \\
&\qquad \int_{\mathbb R_+}\varphi^{\prime}(x)\left[\lambda_1(t)x - \mu_1(t) \right]f_1(x,t)dx \\
\dfrac{d}{dt} \int_{\mathbb R_+}\varphi(x)f_2(x,t)dx &=  \dfrac{\sigma_2^2(t)}2 \int_{\mathbb R_+}\varphi^{\prime\prime}(x)\left(x^{2p}f_2(x,t)\right)dx - \\  &\qquad \int_{\mathbb R_+}\varphi^{\prime}(x)\left[\lambda_2(t)x - \mu_2(t) \right]f_2(x,t)dx,
\end{split}
\end{equation}
for any given smooth test function $\varphi(\cdot)$. 

As early mentioned, the evolution of mean values follow the classical Lotka-Volterra system with logistic growth in \eqref{eq:LV}. The evolution of variances
\[
V_1(t) = \int_{\mathbb R_+}(x-m_1)^2 f_1(x,t)dx, \qquad V_2(t) = \int_{\mathbb R_+}(x-m_2)^2 f_2(x,t)dx, 
\]
can be obtained from \eqref{eq:systFP_weak} and, for any $p\in [1/2,1]$ reads
\begin{equation*}
\begin{split}
\dfrac{d}{dt}V_1(t) &= {\sigma_1^2}(m_1+m_2)m_1^{{(2p)}} - 2\left( \beta m_2 + \dfrac{\alpha}{K}m_1+\alpha\chi\right)V_1, \\
\dfrac{d}{dt}V_2(t) &=\sigma_2^2 m_1 m^{{(2p)}}_2 - 2(\gamma(\mu-m_1) + \nu\theta)V_2, 
\end{split}
\end{equation*}
where $m_1^{(2p)}(t) , m_2^{(2p)}(t)$ are the moments of order $2p \in [1,2]$, namely
\be\label{eq:mom}
m_k^{(2p)}(t) = \int_{\R_+} x^{2p} f_k(x,t) \, dx, \qquad k =1,2.
\ee
Whenever $p = 1/2$. or $p=1$ we obtain an explicit expression for the evolution. If $p = 1/2$ we get
\begin{equation}
\label{eq:var_p12}
\begin{split}
\dfrac{d}{dt}V_1(t) &= {-2\left( \beta m_2(t) + \dfrac{\alpha}{K}m_1(t) + \alpha\chi\right)V_1(t) + \sigma_1^2(m_1(t) + m_2(t))m_1(t)} \\
\dfrac{d}{dt}V_2(t) &=- 2(\gamma(\mu-m_1(t)) + \nu\theta)V_2(t) + \sigma_2^2 m_1(t) m_2(t) , 
\end{split}
\end{equation}
which, provided $\alpha(1+\chi)>0$ and $\gamma\mu - \delta + \nu\theta>0$, admits the stationary value 
\begin{equation}
\label{eq:Vequil_p12}
\mathbf V_{(p = 1/2)} = \left(\dfrac{\sigma_1^2 m_1^\infty(m_1^\infty + m_2^\infty)}{2\alpha(\chi + 1)}, \dfrac{\sigma_2^2 m_1^\infty m_2^\infty}{2(\gamma\mu-\delta + \nu\theta)} \right), 
\end{equation}
{where $m_1^\infty,m_2^\infty$ have been defined in \eqref{eq:equilibrium_mean}.}
On the other hand, if $p = 1$ we obtain
\begin{equation}
\label{eq:var_p1}
\begin{split}
\dfrac{d}{dt}V_1(t) &= \left[\sigma_1^2(m_1(t)+m_2(t)) - 2\left(\beta m_2(t) + \dfrac{\alpha}{K}m_1(t) + \alpha \chi \right)\right]V_1(t)  + \sigma_1^2(m_1(t) + m_2(t) )m_1^2(t) \\
\dfrac{d}{dt}V_2(t) &=\left[\sigma_2^2 m_1(t) - 2(\gamma(\mu-m_1(t)) + \nu\theta)\right]V_2(t) + \sigma_2^2 m_1(t) m_2^2(t).
\end{split}
\end{equation}
In this case one has the stationary values
\begin{equation}
\label{eq:Vequil_p1}
\mathbf V_{(p = 1)} = \left( \dfrac{\sigma_1^2(m_1^\infty + m_2^\infty)(m_1^\infty)^2}{2\alpha(\chi+1) - \sigma_1^2(m_1^\infty + m_2^\infty)}, \dfrac{\sigma_2^2 m_1^\infty (m_2^\infty)^2}{2(\gamma\mu - \delta +\nu\theta) - \sigma_2^2m_1^\infty}\right).
\end{equation}

These values are compatible with the coexistence equilibrium \fer{eq:equilibrium_mean}  provided 
\[
2(\beta m_2^\infty + \frac{\alpha}{K} m_1^\infty + \alpha \chi) - \sigma_1^2(m_1^\infty + m_2^\infty)>0, \qquad 2(\gamma(\mu - m_1^\infty) + \nu\theta) - \sigma_2^2 m_1^\infty>0,
\]
{which, taking into account the equilibrium values \eqref{eq:equilibrium_mean}, imply a bound from above on the diffusion coefficients given by
\[
\sigma_1^2< \dfrac{\delta\beta K + \alpha(\gamma K - \delta)}{2\alpha\beta\gamma K(\chi+1)}, \qquad \sigma_2^2 <\dfrac{2\gamma(\gamma\mu-\delta + \nu\theta)}{\delta}.
\]
We may observe that, under the hypotheses discussed in \cite{BMTZ}, the upper bound for $\sigma_1^2$ is always positive while the one connected to $\sigma_2^2$ is positive under the identification $\delta = \gamma\mu-\nu$.} These explicit results allow to compare the resulting stationary variances, and verify the impact of the parameter $p$. In Figure \ref{fig:mv} we represent the evolution of mean and variance from \eqref{eq:LV}-\eqref{eq:var_p12}-\eqref{eq:var_p1} for $p = 1/2$ and $p = 1$. As expected, we may observe how the dynamics of the macroscopic quantities lead to the emergence of a stable configuration of expected values and variances. 

\begin{figure}
\centering
\includegraphics[scale = 0.24]{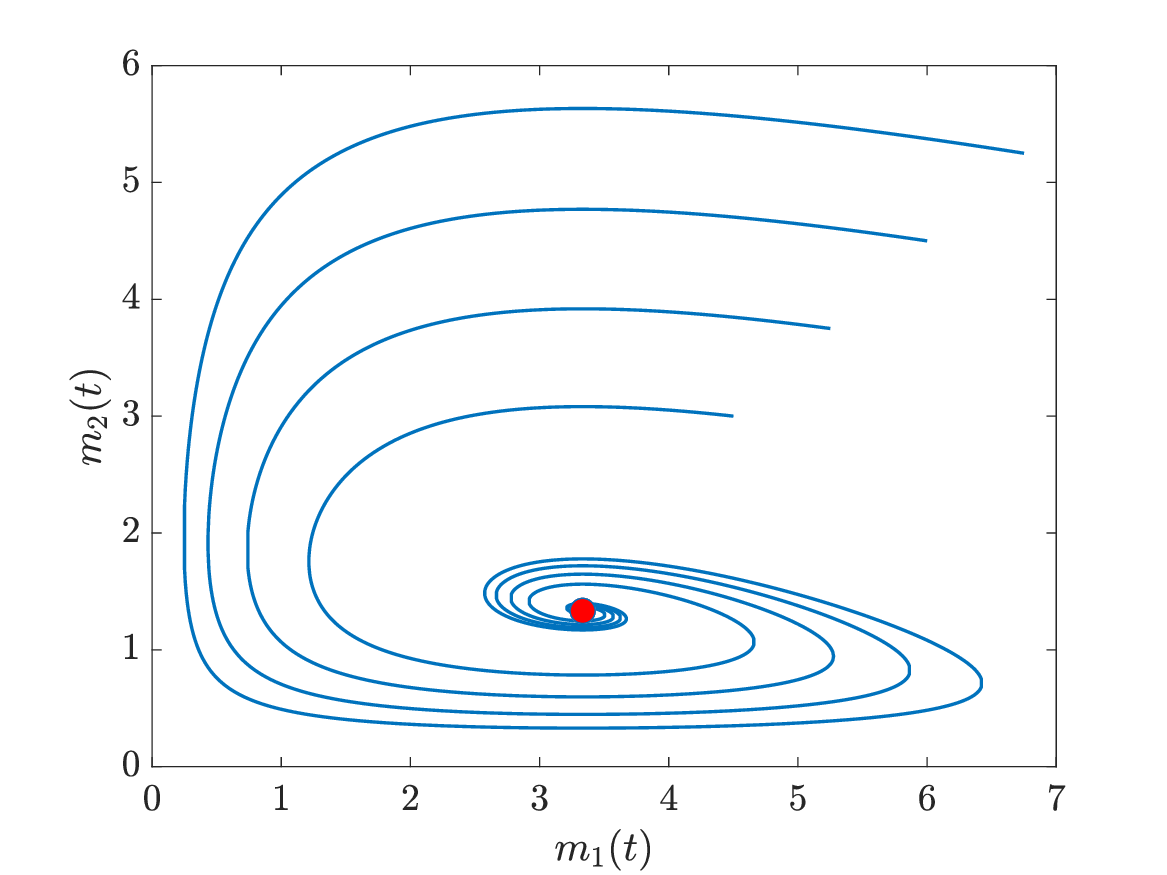}
\includegraphics[scale = 0.24]{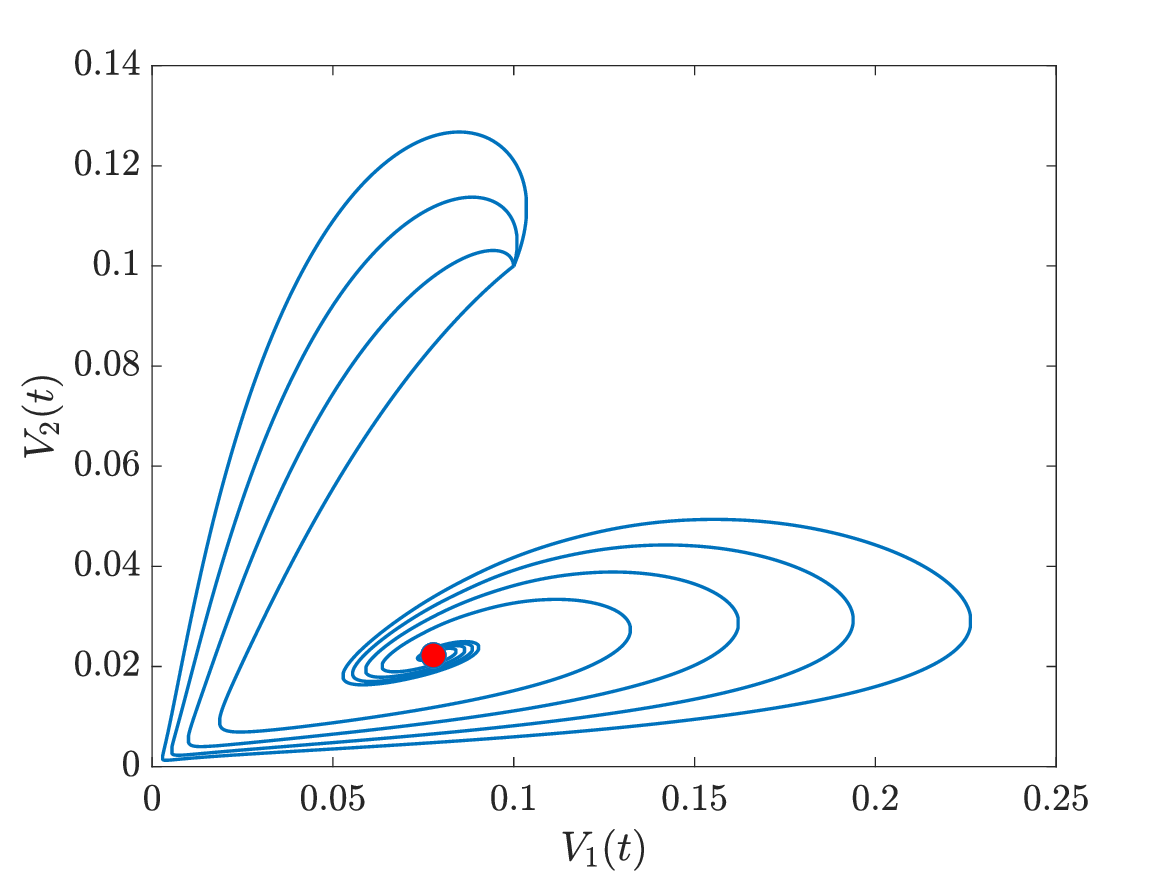}
\includegraphics[scale = 0.24]{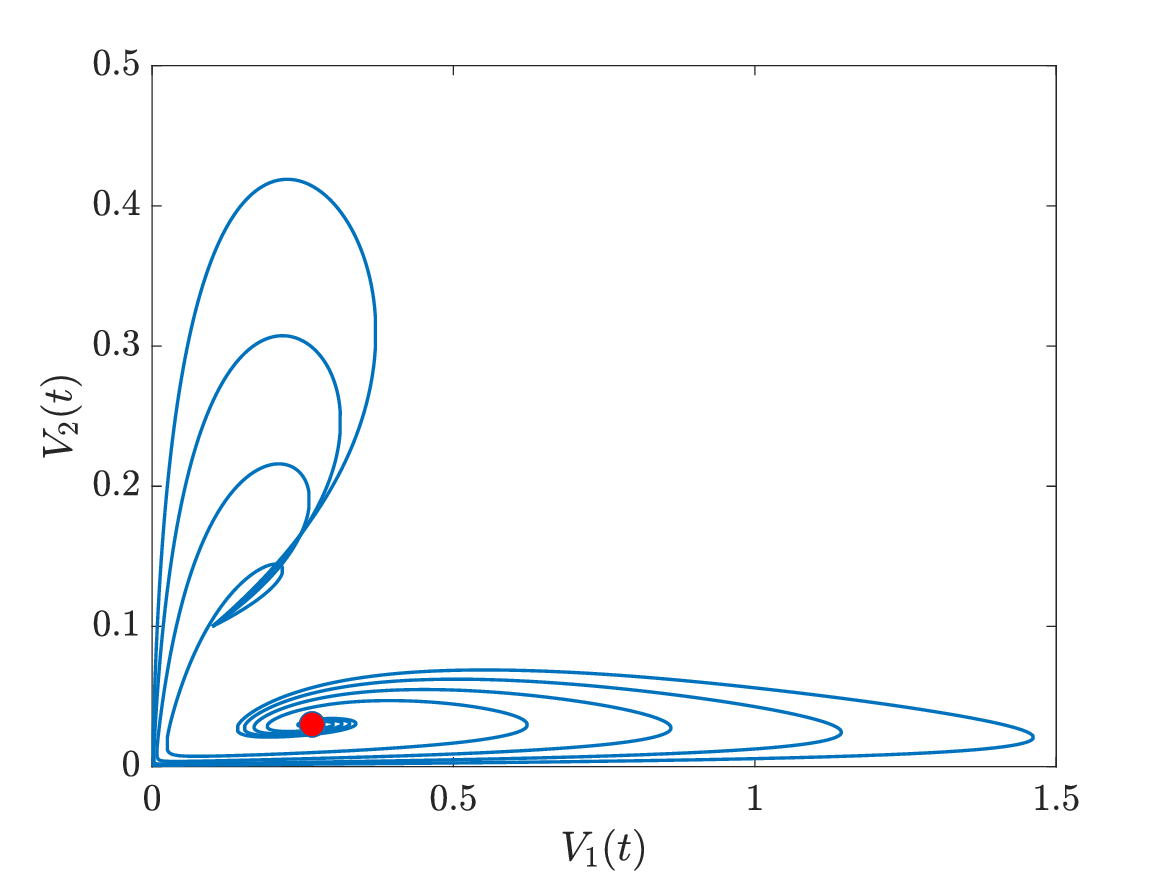}
\caption{Left: convergence of expected values $\mathbf m = (m_1,m_2)$ towards equilibrium $\mathbf m^\infty$ defined in \eqref{eq:equilibrium_mean}. Center: convergence of the variances solution to \eqref{eq:var_p12} obtained in the case $p = 1/2$ towards $\mathbf V^\infty_{(p=1/2)}$ in \eqref{eq:Vequil_p12}. Right: convergence of the variance solution to \eqref{eq:var_p1} obtained in the case $p = 1$ towards $\mathbf V^\infty_{(p=1)}$ in \eqref{eq:Vequil_p1}. The equilibrium points have been highlighted in red, we considered the parameters in Table \ref{Table:params} and as initial conditions $\mathbf m(0)= (\frac{9}{2},\frac 3 4), ( \frac{21}{4},\frac{15}{4}), (\frac{27}{4},\frac{21}{4}), (\frac{15}{2},6)$, and $\mathbf V_{(p = 1/2)} = \mathbf{V}_{(p=1)} = (\frac{1}{10},\frac{1}{10})$.  }
\label{fig:mv}
\end{figure}

\section{Trends to equilibrium}\label{sect:4}
In this section, we aim to study the long-time asymptotic of the Fokker-Planck model \eqref{intro:systFP}, where the coefficients are given by \fer{eq:values}. In more details, we start by  showing that the solution  to a Fokker--Planck type equation like the ones in \fer{intro:systFP}, but with constant coefficients will converge  in Cramér distance towards its equilibrium density in the whole range of the parameter $p$, namely $1/2 \le p \le 1$. Then, we extend the convergence results to cover the case of variable coefficients, which is the main purpose of this paper. Further, we will show that these convergence results can be extended to show decay towards equilibrium in various Energy distances in the border cases $p =1/2$ and $p=1$, which, at variance with the general case, can be treated by resorting to Fourier transform.

\subsection{Quasi-equilibrium distributions}\label{sect:41}
Since the coefficients of system \fer{intro:systFP} depend on time, for any given $t >0$, one can recover particular time-dependent solutions, the quasi-equilibrium densities. These densities $f^q_k(x,t)$, $k = 1,2$ are the probability densities solutions to the first order differential equations

\be\label{eq:qua} 
\dfrac{\sigma_k^2(t)}{2} \dfrac{\partial }{\partial x } (x^{2p} f^q_k(x,t)) +  \left( \lambda_k(t)x  - \mu_k(t)\right) f_k^q(x,t) =0
\ee
It is immediate to show that the unit mass solutions of \fer{eq:qua}, for any $1/2 \le p \le 1$  and $k =1,2$ are the functions
\be\label{eq:gen-Gamma}
f_k^q(x,t) = C_k^p(t)\,x^{-2p}\,
\exp\!\left[
-\frac{2\lambda_k(t)}{\sigma^2_k(t)(2-2p)}\,x^{\,2-2p}
-\frac{2\mu_k(t)}{\sigma^2_k(t)(2p-1)}\,\frac1{x^{\,2p-1}},
\right]
\ee
where $C_1^p(t),C_2^p(t)>0$ are suitable functions of time such that the quasi-equilibria are of unit mass. Note that the probability densities \fer{eq:gen-Gamma} can be rewritten as the product of the densities of two generalised Gamma \cite{Sta}, one with slim tails, the other with fat tails. Only in the limit cases $p\to 1/2$ and $p \to 1$ they reduce to  pure Gamma densities (respectively  inverse Gamma), and only in these limit cases one can explicitly compute the normalisation constants $C_k(t)$, $k=1,2$. 

The interest in the quasi-equilibria is related to the fact that one  expects these solutions to represent well, for intermediate times, the true solution of the system. For this reason, it would be of paramount importance to understand if this property can be verified.

Note that, since for $t \to \infty$ the values in \fer{eq:values} converge to their limit values in \fer{eq:values_large}, the quasi-equilibrium densities converge towards the equilibria
\be\label{eq:equi-Gamma}
f_k^\infty (x) = C_k^p\,x^{-2p}\,
\exp\!\left[
-\frac{2\lambda_{k,\infty}}{\sigma_{k,\infty}^2(2-2p)}\,x^{\,2-2p}
-\frac{2\mu_{k,\infty}}{\sigma^2_{k,\infty}(2p-1)}\,\frac1{x^{\,2p-1}},
\right]
\ee
where $C_1^p,C_2^p>0$ are the constants such that the quasi-equilibria are of unit mass, and the constants $\lambda_{k,\infty}$, $\mu_{k,\infty}$, $\sigma_{k,\infty}^2$ are the ones defined in \eqref{eq:values_large} . As before, only the border cases $p= 1/2$ and $p = 1$ provide a known class of equilibrium densities, Gamma and respectively inverse Gamma densities.  

In the case $p = {1}/{2}$ the quasi-equilibrium  Gamma densities,  take the form
\begin{equation}
\label{eq:quasi_eq_G}
\begin{split}
f_1^q(x,t) &= \dfrac{\omega_1(t)^{-\nu_1(t)}}{\Gamma(\nu_1(t))}x^{\nu_1(t)-1}
\exp\!\left[
-\dfrac{x}{\omega_1(t)}\right], \\
f_2^q(x,t) &= \dfrac{\omega_2(t)^{-\nu_2(t)}}{\Gamma(\nu_2(t))} x^{\nu_2(t)-1}\exp\left[-\dfrac{x}{\omega_2(t)} \right]
\end{split}
\end{equation}
where 
\[
\begin{split}
\omega_1(t) &=  \dfrac{\sigma_1^2(m_1(t)+m_2(t))}{2(\beta m_2(t) + \frac{\alpha}{K}m_1(t) + \alpha \chi)}, \qquad \nu_1 = \frac{2\alpha(\chi+1)m_1(t)}{\sigma_1^2(m_1(t)+m_2(t))} \\
\omega_2(t) &= \dfrac{\sigma_2^2m_1(t)}{2(\gamma(\mu-m_1(t))+\nu\theta)}, \qquad \nu_2 = \dfrac{2\nu(\theta+1)m_2(t)}{\sigma_2^2 m_1(t)}. 
\end{split}
\]
The mean values of the  quasi-equilibrium Gamma density \fer{eq:quasi_eq_G} are
\[
m_1^q(t) = \omega_1(t)\nu_1(t) = \dfrac{\alpha(\chi+1)m_1(t)}{\beta m_2(t) + \frac{\alpha}{K}m_1(t) + \alpha \chi}, \qquad m_2^q(t) = \omega_2(t)\nu_2(t) = \dfrac{\nu(\theta+1)m_2(t)}{\gamma(\mu-m_1(t))+\nu\theta}. 
\]
Note that $f_1^\infty(x),f_2^\infty(x)$ are the Gamma densities
\begin{equation}
\label{eq:equilibrium_G}
\begin{split}
f_1^\infty(x) = \dfrac{(\omega_1^\infty)^{-\nu_1^\infty}}{\Gamma(\nu_1^\infty)}x^{\nu_1^\infty-1}
\exp\!\left[
-\dfrac{x}{\omega_1^\infty}\right], \qquad
f_2^\infty(x) = \dfrac{(\omega_2^\infty)^{-\nu_2^\infty}}{\Gamma(\nu_2^\infty)} x^{\nu_2^\infty-1}\exp\left[-\dfrac{x}{\omega_2^\infty} \right],
\end{split}
\end{equation}
with parameters 
\[
\begin{split}
\omega_1^\infty = \dfrac{\sigma_1^2(m_1^\infty + m_2^\infty)}{2(\beta m_2^\infty + \frac{\alpha}{K} m_1^\infty + \alpha \chi)}, \qquad  \nu_1^\infty = \dfrac{2\alpha(\chi+1)m_1^\infty}{\sigma_1^2(m_1^\infty + m_2^\infty)} \\
\omega_2^\infty = \dfrac{\sigma_2^2 m_1^\infty}{2(\gamma(\mu - m_1^\infty) + \nu \theta)}, \qquad \nu_2^\infty = \dfrac{2\nu(\theta + 1)m_2^\infty}{\sigma_2^2 m_1^\infty},
\end{split}
\] 
Likewise, if $p = 1$ we obtain inverse-Gamma-shaped quasi-equilibria
\begin{equation}
\label{eq:quasi_eq_invG}
\begin{split}
f_1^q(x,t) =  \dfrac{\bar \omega_1^{\bar \nu_1}}{\Gamma(\bar \nu_1)} x^{-1-\bar \nu_1}
\exp\!\left[
-\dfrac{\bar \omega_1}{x}
\right], \qquad
f_2^q(x,t) =\dfrac{\bar \omega_2^{\bar \nu_2}}{\Gamma(\bar \nu_2)} x^{-1-\bar \nu_2}
\exp\!\left[-\dfrac{\bar \omega_2}{x}\right]
\end{split}
\end{equation}
where
\begin{equation}
\label{eq:param_iG}
\begin{split}
\bar \omega_1 &= \frac{2\left(\beta m_2+\frac{\alpha}{K}m_1+\alpha\chi\right)}{\sigma_1^2(m_1+m_2)}, \qquad \bar \nu_1 = \dfrac{2\alpha(\chi+1)m_1}{\sigma_1^2(m_1+m_2)} + 1 \\
\bar \omega_2 &= \dfrac{2(\gamma(\mu-m_1) + \nu\theta)}{\sigma_2^2 m_1} ,\quad \qquad \bar \nu_2 = \dfrac{2\nu(\theta+1)m_2}{\sigma_2^2 m_1} + 1. 
\end{split}
\end{equation}
In this case,  the shape parameters in \eqref{eq:param_iG} are such that $\bar \nu_1>1,\bar \nu_2>1$. Therefore, the first order moment of the inverse Gamma quasi-equilibrium are well defined and given by 
\[
m_1^q = \dfrac{\bar \omega_1}{\bar \nu_1-1} = \dfrac{\beta m_2 + \frac{\alpha}{K}m_1 + \alpha \chi}{\alpha(\chi+1)m_1}, \qquad m_2^q = \dfrac{\bar\omega_2}{\bar \nu_2-1} = \dfrac{\gamma(\mu-m_1)+\nu\theta}{\nu(\theta+1)m_2}
\]
The equilibrium densities if $p =1$ are  inverse Gamma densities
\begin{equation}
\label{eq:equilibrium_invG}
f_1^\infty(x) = \dfrac{(\bar \omega_1^\infty)^{\bar \nu_1^\infty}}{\Gamma(\bar \nu_1^\infty)} x^{-1-\bar \nu_1^\infty}
\exp\!\left[
-\dfrac{\bar \omega_1^\infty}{x}
\right], \qquad
f_2^\infty(x) =\dfrac{(\bar \omega_2^\infty)^{\bar \nu_2^\infty}}{\Gamma(\bar \nu_2^\infty)} x^{-1-\bar \nu_2^\infty}
\exp\!\left[-\dfrac{\bar \omega_2^\infty}{x}\right]
\end{equation}
are inverse Gamma densities with parameters 
\begin{equation}\label{eq:nu_p1}
\begin{split}
\bar \omega_1^\infty &= \frac{2\left(\beta m_2^\infty+\frac{\alpha}{K}m_1^\infty+\alpha\chi\right)}{\sigma_1^2(m_1^\infty+m_2^\infty)}, \qquad \bar \nu_1^\infty = \dfrac{2\alpha(\chi+1)m_1^\infty}{\sigma_1^2(m_1^\infty+m_2^\infty)} + 1 \\
\bar \omega_2^\infty&= \dfrac{2(\gamma(\mu-m_1^\infty) + \nu\theta)}{\sigma_2^2 m_1^\infty} ,\quad \qquad \bar \nu_2^\infty = \dfrac{2\nu(\theta+1)m_2^\infty}{\sigma_2^2 m_1^\infty} + 1, 
\end{split}
\end{equation}

\subsection{Convergence to equilibrium in Cramér distance}\label{sec:equi}

As a first step, in this Section we study directly the convergence of the solutions to the Fokker--Planck system \fer{intro:systFP} towards the equilibrium distributions, for any value of the diffusion constant $1/2 \le p \le 1$.  To start with, we discuss the convergence to equilibrium in Cramér distance, as defined in \fer{Cramer}. In this case,  the equilibria  take the form \fer{eq:equi-Gamma}.

 Since in system \fer{intro:systFP} the coefficients depend on time, the Fokker--Planck equations in \fer{intro:systFP} read
 \begin{equation}
\label{eq:FP-gen}
\dfrac{\partial}{\partial t} f_k(x,t) = \dfrac{\sigma_k^2(t)}{2}\dfrac{\partial^2}{\partial x^2}  (x^{2p}f_k(x,t)) +\dfrac{\partial}{\partial x}  [(\lambda_k(t) x-\mu_k(t))f_k(x,t)], \qquad k=1,2,
\end{equation}
where $\sigma_k^2(t), \lambda_k(t), \mu_k(t)$ are defined in \fer{eq:values}. These linear equations are coupled with no-flux boundary conditions in $x=0$. 
Therefore, in view of the convergence of mean values $m_k(t)$ of the Lotka--Volterrra equations \fer{eq:LV} towards a unique equilibrium point, all these parameters converge in time towards the asymptotic values $\sigma^2_{k,\infty}, \lambda_{k,\infty}, \mu_{k,\infty}$ in \eqref{eq:values_large}.

In reason of \fer{eq:qua}, we know that the equilibrium densities $f^\infty_k(x)$, $k=1,2$ solve the equations
\be\label{eq:chiave}
\dfrac{\sigma^2_{k,\infty}}{2}\dfrac{\partial^2}{\partial x^2}  (x^{2p}f_k^\infty(x)) +\dfrac{\partial}{\partial x}  ((\lambda_{k,\infty} x-\mu_{k,\infty})f^{\infty}_k(x)) =0.
\ee
Therefore, equation \fer{eq:FP-gen} can be rewritten in the equivalent form
\begin{equations} \label{eq:chiave2}
& \dfrac{\partial}{\partial t}[ f_k(x,t) -f^\infty_k(x)] = \dfrac{\sigma^2_k(t)}{2}\dfrac{\partial^2}{\partial x^2}  [x^{2p}(f_k(x,t)-f^\infty_k(x))] + \\
&\qquad \dfrac{\partial}{\partial x} [ (\lambda_k(t) x-\mu_k(t))(f_k(x,t)-f^\infty_k(x))] +\\
&\qquad\qquad\dfrac{\sigma^2_k(t)-\sigma^2_{k,\infty}}{2}\dfrac{\partial^2}{\partial x^2}  (x^{2p}f^\infty_k(x)) + 
\dfrac{\partial}{\partial x} \left[ (\lambda_k(t)-\lambda_{k,\infty}) x-( \mu_k(t)-\mu_{k,\infty})]f_k^\infty(x)\right].
\end{equations}
Furthermore, since the equilibrium density satisfies \fer{eq:chiave}, it holds
\[
\dfrac{\partial^2}{\partial x^2}  (x^{2p}f_k^\infty(x)) = -\frac {2\lambda_{k,\infty}}{\sigma^2_{k,\infty}} \dfrac{\partial}{\partial x} ( xf_k^\infty(x))  + \frac {2\mu_{k,\infty}}{\sigma^2_{k,\infty}}\dfrac{\partial}{\partial x} ( f_k^\infty(x)), \qquad k=1,2,
\]
and we can eliminate the dependence from the second derivative of $f_k^\infty$ into the last line of equation \fer{eq:chiave2} to obtain
\begin{equations} \label{eq:chiave3}
& \dfrac{\partial}{\partial t}[ f_k(x,t) -f_k^\infty(x)] = \dfrac{\sigma_k^2(t)}{2}\dfrac{\partial^2}{\partial x^2}  [x^{2p}(f_k(x,t)-f_k^\infty(x))] + \\
& \dfrac{\partial}{\partial x} [ (\lambda_k(t) x-\mu_k(t))(f_k(x,t)-f_k^\infty(x))] +
A_k(t) \dfrac{\partial}{\partial x}[xf_k^\infty(x)]  + B_k(t) \dfrac{\partial}{\partial x}f_k^\infty(x),
\end{equations}
where
\be\label{eq:AB}
A_k(t) = (\lambda_k(t) -\lambda_{k,\infty}) -\lambda_{k,\infty}\frac{\sigma_k^2(t) -\sigma_{k,\infty}^2}{\sigma_{k,\infty}^2}; \quad 
B_k(t) = \mu_{k,\infty}\frac{\sigma_k^2(t) -\sigma_{k,\infty}^2}{\sigma_{k,\infty}^2}- ( \mu_k(t) -\mu_{k,\infty})
\ee
with $k=1,2$. Integrating equation \fer{eq:chiave3} in $[0,x]$, and recalling the definition of cumulative function,  as defined in \fer{cumu},  gives the evolution of the difference between the cumulative functions of the solution to the Fokker--Planck equation \fer{eq:FP-gen},  and $F^\infty_k(x)$, its equilibrium solution. This integration gives
\begin{equations} \label{eq:cum}
& \dfrac{\partial}{\partial t}[ F_k(x,t) -F_k^\infty(x)] = \dfrac{\sigma_k^2(t)}{2}\dfrac{\partial}{\partial x}  [x^{2p}(f_k(x,t)-f_k^\infty(x))] + \\
&  (\lambda_k(t) x-\mu_k(t))\dfrac{\partial}{\partial x}(F_k(x,t)-F_k^\infty(x))+
A_k(t) xf_k^\infty(x)  + B_k(t)f_k^\infty(x).
\end{equations}
This equation allows to easily compute the evolution in time of Cram\'er distance $d(f_k(t),f_k^\infty)$, as defined in \fer{Cramer}. We obtain
\begin{equations}\label{eq:Cramer}
&\frac d{dt}\int_0^\infty (F_k(x,t) -F_k^\infty(x))^2 \, dx = \sigma^2_k(t)\int_0^\infty (F_k(x,t) -F_k^\infty(x))\dfrac{\partial}{\partial x}  [x^{2p}(f_k(x,t)-f_k^\infty(x))]  \,dx + \\
& \int_0^\infty  (\lambda_k(t) x-\mu_k(t))\dfrac{\partial}{\partial x}[(F_k(x,t)-F_k^\infty(x))^2]\, dx  +\\
& 2\int_0^\infty (F_k(x,t)-F_k^\infty(x))\left[A_k(t) xf_k^\infty(x)  + B_k(t)f_k^\infty(x) \right] 
\end{equations}
for $k=1,2$. Integrating by parts the first two integrals, and applying Cauchy-Schwartz inequality to the last two, we conclude with the differential inequality 
\begin{equations}\label{eq:Cramer}
&\frac d{dt}\int_0^\infty (F_k(x,t) -F_k^\infty(x))^2 \, dx \le \\
&- \sigma_k^2(t)\int_0^\infty x^{2p}(f_k(x,t)-f_k^\infty(x))^2  \,dx - \lambda_k(t) \int_0^\infty (F_k(x,t)-F_k^\infty(x))^2\, dx + \\
& 2A_k(t) \left(\int_0^\infty (F_k(x,t)-F_k^\infty(x))^2\, dx\right)^{1/2} \left(\int_0^\infty x^2(f_k^\infty(x))^2 \,dx\right)^{1/2} + \\
&2B_k(t)\left(\int_0^\infty (F_k(x,t)-F_k^\infty(x))^2\, dx\right)^{1/2} \left(\int_0^\infty (f_k^\infty(x))^2 \,dx\right)^{1/2}.
\end{equations}
It is important to observe that, in view of definition \fer{eq:equi-Gamma} of the equilibrium densities, which have finite mean values, $f_k^\infty(x)$ and $xf_k^\infty(x)$ belong to $L^2(\R_+)$, the worst case being $p=1$. Consequently, for every $p$ with $1/2 \le p \le 1$, the Cram\'er distance \eqref{Cramer} satisfies the inequality
\begin{equations}\label{eq: ine-Cra}
\frac d{dt} d(f_k(t),f_k^\infty) &\le -\lambda_k(t) d(f_k(t),f_k^\infty) +\\
& 2 A_k(t) d^{1/2}(f_k(t),f_k^\infty) \|xf_k^\infty(x)\|_{L^2} + 2 B_k(t) d^{1/2}(f_k(t),f_k^\infty) \|f_k^\infty(x)\|_{L^2}.
\end{equations}
This differential inequality is a first order Bernoulli equation and it can be explicitly solved  to prove convergence towards equilibrium, provided 
 $A_k(t)$ and $B_k(t)$ converge to zero. In particular, the solution satisfies the bound
\[
d(f_k,f_k^\infty) \le  \left(\sqrt{d_0}e^{-\frac 1 2\int_0^t \lambda_k(\tau)d\tau}  + \int_0^t e^{-\frac 1 2\int_s^t \lambda_k(s)ds}\left(A_k(s)\| x f_k^\infty\|_{L^2} + B_k(s)\|f_k^\infty \|_{L^2}\right)ds\right)^2,
\]
with $d_0 = d(f_k(0),f_k^\infty)$. Therefore, provided $\lambda_k(t) \ge \lambda^*>0$ and 
\be\label{Mt}
M(t) = A_k(t)\| x f_k^\infty\|_{L^2} + B_k(t)\|f_k^\infty \|_{L^2} \to 0
\ee
exponentially, we have
\[
d(f_k,f_k^\infty) \le \left( \sqrt{d( f_{k}(0) , f_k^\infty)}e^{-\frac{\lambda^*}{2}t} + \dfrac{2M(t)}{\lambda^*}(1-e^{-\frac{\lambda^*}{2}t})\right)^2,
\]
from which we deduce exponential decay of the initial Cramér distance. 
Hence we proved

\begin{theorem}\label{th:p-gen}
Let $f_k(x,t)$, $k=1,2$ be the solutions of the Fokker--Planck equations \fer{eq:FP-gen}, with initial values the probability densities $f_k(0)=f_{k,0}$  such that, for any $1/2\le p \le 1$ 
\[
d( f_{k,0} , f_k^\infty)  < +\infty, \qquad  k = 1,2, 
\]
where $f_k^\infty$ are the equilibrium solutions in \fer{eq:equi-Gamma}. Then, the solutions to \fer{eq:FP-gen}, and consequently of system \eqref{intro:systFP} satisfy
\[
 d( f_{k} (t), f_k^\infty )\le \left( \sqrt{d( f_{k,0} , f_k^\infty)}e^{-\frac{\lambda^*}{2}t} + \dfrac{2M(t)}{\lambda^*}(1-e^{-\frac{\lambda^*}{2}t})\right)^2 \qquad  k = 1,2,
\]
where $M(t) \to 0$ as $t\to+\infty$ {and $\lambda_k(t)\ge \lambda^*>0$, $k=1,2$, for all $t\ge0$. }
\end{theorem}

\begin{remark} It is important to note that the advantage of working with Cram\'er distance (namely with the Energy distance of order $\ell=1$) is linked to the possibility in this case to have, for any $1/2\le p\le 1$, a negative contribution $D_k(t)$ from the diffusion term, given in equation \fer{eq:Cramer} by
\[
D_k(t) = - \sigma^2_k(t)\int_0^\infty x^{2p}(f_k(x,t)-f_k^\infty(x))^2  \,dx.
\]
This obstacle is not present in the two extremal cases $p=1/2$ and $p=1$, where we can treat the same problem by resorting to Fourier transform. Hence, in these extremal cases, we can easily extend the convergence to equilibrium in Energy distance of order $\ell$, with $1 <\ell<3/2$ by treating equation \fer{eq:chiave2} in Fourier to show the non positivity of the contribution of the diffusion in this range of the parameter $\ell$. 
\end{remark} 

{
\begin{remark}
From \eqref{eq:equilibrium_mean} we conclude that $\lambda_{1,\infty},\lambda_{2,\infty}>0$. Since $m_1(t) \to m_1^\infty$, $m_2(t) \to m_2^\infty$ as $t \to +\infty$, there exists $t^*>0$ such that the condition $\lambda_1(t),\lambda_2(t)\ge \lambda^*>0$ is true for any $t>t^*$.
\end{remark}
}

\subsection{Convergence in Energy distance for $p=1/2$ and $p=1$}\label{sec:quasi}

In this Section, we will prove that in the relevant cases $p=1/2$ and $p=1$ the solutions $f_k(x,t)$, $k=1,2$, to the Fokker--Planck equations \fer{eq:FP-gen},  converge to their equilibria {$f_k^\infty(x)$}, $k=1,2$ in Energy distance  of order $\ell$, provided the parameter $\ell$ characterising the normalised Energy distance \fer{eq:norm} is sufficiently big (typically $\ell \ge1$). Indeed, the special cases $p=1/2$ and $p=1$ can be treated by writing the Fokker--Planck equations in Fourier transform, which allows {direct computations. In the following we first study the case in which coefficients do not depend on time}. 

\subsubsection{{The time-independent case}} \label{subsect:time_in}
If the coefficients do not depend on time, i.e. $\sigma^2,\lambda,\mu>0$ are time-independent positive parameter, the Fokker--Planck equations in \fer{eq:FP-gen} belong to the class
\begin{equation}
\label{eq:FP_protot_p12}
\dfrac{\partial}{\partial t} f(x,t) = \dfrac{\sigma^2}{2}\dfrac{\partial^2}{\partial x^2}  (x^{2p}f(x,t)) +\dfrac{\partial}{\partial x}  ((\lambda x-\mu)f(x,t)). 
\end{equation}
This linear equation is coupled with no-flux boundary conditions in $x=0$. When $p =1/2$, the steady state is the Gamma distribution
\begin{equation}
\label{eq:finf_FP_protot_p12}
f^\infty(x) = \dfrac{\omega^{-\nu}}{\Gamma(\nu)} x^{\nu-1} e^{-x/\omega}, \qquad \nu = \dfrac{2\mu}{\sigma^2},\omega = \dfrac{\sigma^2}{2\lambda}. 
\end{equation}
We prove
\begin{lemma}\label{lem:energy}
Let $f(x,t)$ be the solution to the Fokker--Planck equation \fer{eq:FP_protot_p12}  with $p =1/2$, initial value $f_0(x) $ and equilibrium density $f^\infty(x)$, as given in \fer{eq:finf_FP_protot_p12}, such that $\mathcal{E}_\ell (f_0, f_\infty)$ is bounded for $1\le \ell <3/2$. Then the Energy distance $\mathcal{E}_\ell (f(t), f_\infty)$ is exponentially decaying to zero at rate $\lambda(2\ell-1)$.
\end{lemma}

\begin{proof}
In what follows we assume that the initial value $f_0(x)$ belong to $L^2(\R_+)\cap \dot{H}_{-1}(\R_+)$.  This implies that the solution $f(x,t) \in L^2(\R_+)$ for any subsequent time $t>0$. Moreover, as proven in Section \ref{sec:equi}, under the same hypotheses on the initial value, there is exponential convergence in Cram\'er distance of the solution to \fer{eq:FP_protot_p12} towards the steady state \fer{eq:finf_FP_protot_p12}. 

The result  of Section \ref{sec:equi} can be generalized by resorting to the possibility to make use in this case of the Fourier transform. 
The Fourier transformed version of the Fokker-Planck equation \fer{eq:FP_protot_p12} reads
\[
\dfrac{\partial}{\partial t}  \hat f(\xi,t) = -\dfrac{\sigma^2}{2}i\xi^2 \dfrac{\partial}{\partial \xi} \hat f(\xi,t) - \lambda \xi \dfrac{\partial}{\partial \xi} \hat f(\xi,t) - i\mu \xi \hat f(\xi,t). 
\]
Let us rewrite $\hat f(\xi,t) $ as follows $\hat f(\xi,t) = a(\xi,t) + ib(\xi,t)$, where $ a(\xi,t) = \mathrm{Re}(\hat f)$ and $b(\xi,t) = \mathrm{Im}(\hat f)$. Then we have
\[
\begin{split}
\dfrac{\partial a(\xi,t)}{\partial t} = \dfrac{\sigma^2}{2}\xi^2 \dfrac{\partial}{\partial \xi}b(\xi,t) - \lambda \xi \dfrac{\partial}{\partial \xi} a(\xi,t) + \mu \xi b(\xi,t) \\
\dfrac{\partial b(\xi,t)}{\partial t} = -\dfrac{\sigma^2}{2}\xi^2 \dfrac{\partial}{\partial \xi} a(\xi,t) - \lambda \xi \dfrac{\partial}{\partial \xi} b(\xi,t) - \mu\xi a(\xi,t). 
\end{split}
\]
To compute the evolution of $|\hat f(\xi)|^2 =a^2(\xi)+b^2(\xi) $ we multiply the above equation by $2a$ and $2b$, respectively, to obtain
\begin{equation}
\label{eq:evo_a2_b2}
\begin{split}
\dfrac{\partial a^2(\xi,t)}{\partial t} =  \sigma^2\xi^2 a(\xi,t)\dfrac{\partial}{\partial \xi}b(\xi,t) - \lambda \xi  \dfrac{\partial}{\partial \xi} a^2(\xi,t) + 2\mu \xi  a(\xi,t)b(\xi,t) \\
\dfrac{\partial b^2(\xi,t)}{\partial t} =-\sigma^2\xi^2 b(\xi,t)\dfrac{\partial}{\partial \xi} a(\xi,t) - \lambda \xi \dfrac{\partial}{\partial \xi} b^2(\xi,t) -2\mu\xi a(\xi,t) b(\xi,t). 
\end{split}
\end{equation}
Therefore, if $\ell = 1$, from \eqref{eq:evo_a2_b2} we have
\begin{equation}\label{eq:p1_el1}
\dfrac{\partial}{\partial t}\int_{\mathbb R}|\xi|^{-2}|\hat f|^2 d\xi = \sigma^2 \int_{\mathbb R}  \left[ a(\xi,t)\dfrac{\partial}{\partial \xi}b(\xi,t)  -  b(\xi,t)\dfrac{\partial}{\partial \xi} a(\xi,t) \right]d\xi -\lambda \int_{\mathbb R}|\xi|^{-2} \xi \dfrac{\partial }{\partial \xi} |\hat f|^2d\xi. 
\end{equation}
Next, we observe that being $f(x,t) $ a probability density in $L^2(\R_+)$ with finite variance, thanks to the Plancherel formula we get
\begin{equation}
\label{eq:sign_F}
0\le \int_{\mathbb R_+} x f^2(x,t)dx = \dfrac{1}{{2\pi}}\int_{\mathbb R} \widehat{xf}(\xi,t) \overline{\widehat{f}}(\xi,t)d\xi.
\end{equation}
Since
\[
\widehat{xf}(\xi,t) = \int_{\mathbb R} xf(x,t) e^{-i\xi x}dx = i \dfrac{\partial \hat f}{\partial \xi}(\xi,t),
\]
and $ \overline{\widehat{f}}(\xi,t) = a(\xi,t) - ib(\xi,t)$,  
\[
\begin{split}
 &\int_{\mathbb R_+} x f^2(x,t)dx = \int_{\mathbb R} \dfrac{\partial}{\partial \xi}[ia(\xi,t) - b(\xi,t)](a(\xi,t) - ib(\xi,t)) d\xi = \\
 & \int_{\mathbb R}\left[-a(\xi,t)\dfrac{\partial b(\xi,t)}{\partial \xi} + b(\xi,t)\dfrac{\partial a(\xi,t)}{\partial \xi} + \dfrac{i}{2}\dfrac{\partial}{\partial \xi}(a^2(\xi,t) + b^2(\xi,t)) \right]d\xi = \\
 & \int_{\mathbb R}\left[-a(\xi,t)\dfrac{\partial b(\xi,t)}{\partial \xi} + b(\xi,t)\dfrac{\partial a(\xi,t)}{\partial \xi}\right]d\xi.
\end{split}\]
On the other hand, since $f(x,t) \in L^2(\R_+)$
\[
\int_{\mathbb R} \dfrac{\partial}{\partial \xi}(a^2(\xi,t) + b^2(\xi,t)) \, d\xi = \left[ a^2(\xi,t) + b^2(\xi,t) \right]_{-\infty}^{+\infty} =0
\]
Therefore,  \eqref{eq:sign_F} implies
\[
\int_{\mathbb R}\left[a(\xi,t)\dfrac{\partial b(\xi,t)}{\partial \xi} - b(\xi,t)\dfrac{\partial a(\xi,t)}{\partial \xi}\right]d\xi \le 0, 
\]
and this shows that  the contribution of the diffusion in \eqref{eq:p1_el1} is negative.  

Now, it is enough to remark that, since the Fokker--Planck equation \fer{eq:FP_protot_p12} is linear, the previous arguments continue to hold if we substitute $f(x,t)$ with the difference $f(x,t) -f^\infty(x)$, where $f^\infty(x)$ is the equilibrium solution of the Fokker--Planck equation \fer{eq:FP_protot_p12}, and trivially satisfies it. Hence we proved that, provided  $f_0(x)-f^\infty(x) \in L^2(\R_+)\cap \dot{H}_{-1}(\R_+)$, the evolution in time of the Energy distance $\mathcal{E}_1(f(t),f^\infty)$ is such that
\[
\begin{split}
\dfrac{d}{dt} \mathcal{E}_1(f(t),f^\infty) \le -\lambda\, \mathcal{E}_1(f(t),f^\infty),
\end{split}\]
which shows exponential convergence to equilibrium of the solution to the Fokker--Planck equation \fer{eq:FP_protot_p12} at the explicit rate $\lambda$, thus confirming the result in \cite{AT}. 

More generally, if $1<\ell<\frac 32$, let  $h_0(t) =f_0(x)-f^\infty(x) \in L^2(\R_+)\cap \dot{H}_{-\ell}(\R_+)$, and let us consider the function $h*k_{2l-2}(x)$, where as usual the symbol $*$ denotes convolution and, for  $0<r<1$ 
\be\label{eq:Riesz} 
\kappa_{r}(x) = \frac 1{c_r}\frac 1{ |x|^{1-r}}; \qquad c_r = \sqrt \pi 2^r \frac{\Gamma\left(\frac r2 \right)}{\Gamma\left(\frac{1- r}2 \right)}
\ee
 is the Riesz potential \cite{Lieb,Stein}, which is such that 
\be\label{rf}
\widehat{h*\kappa_r}{(\xi)} = \frac 1{|\xi|^{r}} \hat h(\xi).
\ee
Since in the interval $1<\ell<\frac 32 $ the value $r=l-1$ is such that $0<r<1$, the Riesz potential is well defined and, owing to \fer{rf}, Plancherel formula gives
\begin{equation}
\label{eq:Riesz1}
\int_{\mathbb R_+}  \left[ h(x,t) * \kappa_{\ell-1}(x)\right]^2 dx = \dfrac{1}{2\pi} \int_{\mathbb R} \frac 1{|\xi|^{2l-2}} |\hat h(\xi)|^2 d\xi = \|h\|_{ \dot{H}_{-(2l-2)}}.
\end{equation}
In view of the previous result in $\dot{H}_{-1}(\R_+)$, and thanks to Lemma \ref{lemma1} , we conclude that, since $0<2l-2<1$, $\|h\|_{ \dot{H}_{-(2l-2)}}$ is bounded. Therefore, for any given $t >0$, the function 
\[
0 \le g(x) = \frac{ \left[ h(x,t) * \kappa_{\ell-1}(x)\right]^2}{ \int_\R \left[ h(x,t) * \kappa_{\ell-1}(x)\right]^2\, dx}
\]
is a probability density such that, since $h(x,t) = 0$ if $x \le 0$, its mass on $\R_+$ is strictly bigger than its mass in $\R_-$, which implies
\be\label{eq:pm}
0 \le \int_\R xg(x)\,dx. 
\ee
Hence
\begin{equation}
\label{eq:inequality_ellg1}
\int_{\mathbb R_+} x \left[ h(x,t) *\kappa_{\ell-1}(x)\right]^2 dx = \dfrac{1}{2\pi} \int_{\mathbb R} \widehat{x h*\kappa_{\ell-1}}(\xi,t) \overline{\widehat{f*\kappa_{\ell-1}}} d\xi \ge 0.
\end{equation}
Since 
\[
 \widehat{x h*\kappa_{\ell-1}}(\xi,t) = \int_{\mathbb R_+}xh*\kappa_{\ell-1}(x,t)e^{-i\xi x}dx = \dfrac{i}{|\xi|^{\ell-1}}\dfrac{\partial \hat h}{\partial \xi},
 \]
  if we set $\hat h(\xi,t) = a(\xi,t) + i b(\xi,t)$, the integral in \eqref{eq:inequality_ellg1} can be written in the form
 \[
 \int_{\mathbb R} \dfrac{1}{|\xi|^{2\ell-2}} \dfrac{\partial}{\partial \xi}[ia(\xi,t) -b(\xi,t) ]\left( a(\xi,t) - ib(\xi,t) \right) d\xi =  \int_{\mathbb R}\left[ - \dfrac{a(\xi,t)}{|\xi|^{2\ell-2}}\dfrac{\partial b(\xi,t)}{\partial \xi} +  \dfrac{b(\xi,t)}{|\xi|^{2\ell-2}}\dfrac{\partial a(\xi,t)}{\partial \xi}   \right]\, d\xi.
 \] 
 Furthermore, since
 \begin{equations}
 &\int_{\mathbb R}\left[ - \dfrac{a(\xi,t)}{|\xi|^{2\ell-2}}\dfrac{\partial b(\xi,t)}{\partial \xi} +  \dfrac{b(\xi,t)}{|\xi|^{2\ell-2}}\dfrac{\partial a(\xi,t)}{\partial \xi}   \right]\, d\xi =\\
& \int_{\mathbb R} \left[ - \dfrac{a(\xi,t)}{|\xi|^{\ell-1}} \frac{\partial }{\partial \xi}\left(\frac{b(\xi,t)}{|\xi|^{\ell-1}}\right) +  \dfrac{b(\xi,t)}{|\xi|^{\ell-1}}\frac{\partial }{\partial \xi}\left(\frac{a(\xi,t)}{|\xi|^{\ell-1}}\right)  \right]\, d\xi.
 \end{equations}
 \eqref{eq:inequality_ellg1} yields
 \[
   \int_{\mathbb R} \left[ - \dfrac{a(\xi,t)}{|\xi|^{\ell-1}} \frac{\partial }{\partial \xi}\left(\frac{b(\xi,t)}{|\xi|^{\ell-1}}\right) +  \dfrac{b(\xi,t)}{|\xi|^{\ell-1}}\frac{\partial }{\partial \xi}\left(\frac{a(\xi,t)}{|\xi|^{\ell-1}}\right)  \right]\, d\xi \le 0
 \]
Hence, if $1<\ell<\frac 32$, from \eqref{eq:evo_a2_b2} we get
\begin{equations}\label{eq:decay12}
\dfrac{\partial}{\partial t} \int_{\mathbb R} |\xi|^{-2\ell} |\hat h|^2 d\xi &= \sigma^2 \int_{\mathbb R}|\xi|^2 \left[\dfrac{a(\xi,t)}{|\xi|^{2\ell}}\dfrac{\partial b(\xi,t )}{\partial \xi} - \dfrac{b(\xi,t)}{|\xi|^{2\ell}}\dfrac{\partial a(\xi,t) }{\partial \xi} \right]d\xi - \lambda \int_{\mathbb R} |\xi|^{-2\ell} \xi \dfrac{\partial}{\partial \xi}|\hat h|^2 d\xi \le \\ 
&\lambda \int_{\mathbb R} |\xi|^{-2\ell} \xi \dfrac{\partial}{\partial \xi}|\hat h|^2 d\xi =  -\lambda(2\ell-1) \int_{\mathbb R} |\xi|^{-2\ell} |\hat h|^2 d\xi. 
\end{equations}
namely exponential convergence towards equilibrium in Energy distance of order $\ell$, when $1<\ell<\frac 32$,   at the explicit rate $\lambda(2\ell-1)$. This concludes the proof of the lemma.
\end{proof}

If $p=1$, as treated in \cite{MTZ,TT}, the Fokker-Planck equation \fer{eq:FP_protot_p12}
has a steady state which is  the inverse Gamma density
\begin{equation}
\label{eq:finf_FP_protot_p1}
f^\infty(x)  = \dfrac{\omega^\nu}{\Gamma(\nu)}x^{-1-\nu}\exp{\left[ -\frac{\omega}{x}\right]}, \qquad \nu = 1+\dfrac{2\lambda}{\sigma^2}, \omega = \frac{2\mu}{\sigma^2}.  
\end{equation}
In this case  the following lemma holds
\begin{lemma}\label{lem:energy2}
Let $f(x,t)$ be the solution to the Fokker--Planck equation \fer{eq:FP_protot_p12} with $p=1$, initial value $f_0(x) $ and equilibrium density $f^\infty(x)$, as given in \fer{eq:finf_FP_protot_p1}, such that $\mathcal{E}_\ell (f_0, f_\infty)$ is bounded for $1/2< \ell <3/2$. Then the Energy distance $\mathcal{E}_\ell (f(t), f_\infty)$ is exponentially decaying to zero at rate $\left[\sigma^2\frac{3-2\ell}4 +\lambda\right](2\ell-1)$.
\end{lemma}

\begin{proof}

Let $h(x,t) = f(x,t) -f^\infty(x)$.
The Fourier transformed version of this equation reads
\[
\dfrac{\partial}{\partial t}\hat h(x,t) = \dfrac{\sigma^2}{2} \xi^2 \dfrac{\partial^2}{\partial \xi^2} \hat h(\xi,t) - \lambda \xi \dfrac{\partial^2}{\partial \xi}\hat h(\xi,t) - i\mu \xi \hat h(\xi,t), 
\]
which, provided  the initial value $h_0(x) \in L^2(\R_+)\cap \dot{H}_\ell(\R_+)$  gives, proceeding as in \cite{MTZ,TT}, the  inequality
\[
\dfrac{\partial}{\partial t}\int_{\mathbb R}|\xi|^{-2\ell}|\hat h|^2d\xi \le -(2\ell-1)\left[\sigma^2 \dfrac{3-2\ell}{4} + \lambda \right] \int_{\mathbb R}|\xi|^{-2\ell}|\hat h|^2 d\xi. 
\]
Therefore, when $1/2< \ell < 3/2 $, one has exponential convergence of the solution towards the inverse Gamma steady state defined in \eqref{eq:finf_FP_protot_p1} in the Energy distance $\mathcal{E}_\ell$. 
\end{proof}

\subsubsection{{The time dependent case}}
Let us now pass to consider, for $p=1/2$, the Fokker--Planck equations \fer{eq:FP-gen}, and let us evaluate the time evolution of the Energy distance of order $\ell$, with $1<\ell<3/2$ to the Fourier transformed version of equation \fer{eq:chiave3}, where $A(t)$ and $B(t)$ are given as in \fer{eq:AB}. We recall that 
\be\label{eq:fou12}
\int_\R \frac{\partial f_k^\infty(x)}{\partial x} e^{-i\xi x}\, dx = i\xi \hat f_k^\infty(\xi); \quad \int_\R \frac{\partial }{\partial x} (x\, f_k^\infty(x))e^{-i\xi x}\, dx = -\xi \frac{\partial\hat f_k^\infty(\xi)}{\partial\xi}.
\ee
It is immediate to verify that the conclusions of Lemma \ref{lem:energy} remain valid even if the coefficients of the Fokker--Planck equation are time dependent, and that in this range of $\ell$ the contribution of the diffusion term is non positive. Hence, {the results of the time-independent case provided in Section \ref{subsect:time_in} lead to the following inequality if $p= 1/2$}
\begin{equations}\label{eq:p1}
\frac d{dt} \int_\R \frac{|\hat f_k(\xi) -\hat f_k^\infty(\xi)|^2}{|\xi|^{2\ell}}\, d\xi &\le -{\lambda_k}(t)(2\ell-1) \int_\R \frac{|\hat f_k(\xi) -\hat f_k^\infty(\xi)|^2}{|\xi|^{2\ell}}\, d\xi +\\
&  2{A_k}(t)  \int_\R \frac{|\hat f_k(\xi) -\hat f_k^\infty(\xi)|}{|\xi|^{2\ell}}\,\left| \xi \frac{\partial\hat f_k^\infty(\xi)}{\partial\xi}\right|\,d\xi +\\
&2 {B_k}(t)  \int_\R \frac{|\hat f_k(\xi) -\hat f_k^\infty(\xi)|}{|\xi|^{2\ell}}\,  \left|\xi \hat f_k^\infty(\xi)\right|\,d\xi.
\end{equations}
To conclude, it is enough to determine a bound for the integrals on the second and third line of \fer{eq:p1}.  By Cauchy-Schwartz inequality we have
\begin{equations}\label{eq:cs}
 & \int_\R \frac{|\hat f_k(\xi) -\hat f_k^\infty(\xi)|}{|\xi|^{2\ell}}\,  |\xi \hat f_k^\infty(\xi)|\,d\xi  =
  \int_\R \frac{|\hat f_k(\xi) -\hat f_k^\infty(\xi)|}{|\xi|^{\ell}}\,  \frac{|\hat f_k^\infty(\xi)|}{|\xi|^{\ell-1}}\,d\xi \le \\
  &\left( \int_\R \frac{|\hat f_k(\xi) -\hat f_k^\infty(\xi)|^2}{|\xi|^{2\ell}}\,d\xi\right)^{1/2}  \left( \int_\R \frac{|\hat f_k^\infty(\xi)|^2}{|\xi|^{2\ell-2}}\,d\xi \right)^{1/2} .
\end{equations}
On the other hand, recalling that $f_k^\infty$ is a probability density, so that  $|\hat f_k^\infty(\xi)| \le 1$,  for any given constant $R>0$ we have
\begin{equations}\label{eq:stima}
 & \int_\R \frac{|\hat f_k^\infty(\xi)|^2}{|\xi|^{2\ell-2}}\,d\xi \le   \int_{|\xi| \le R} \frac{|\hat f_k^\infty(\xi)|^2}{|\xi|^{2\ell-2}}\,d\xi + \int_{|\xi| > R} \frac{|\hat f_k^\infty(\xi)|^2}{|\xi|^{2\ell-2}}\,d\xi \le \\ 
 &  \int_{|\xi| \le R} \frac{1}{|\xi|^{2\ell-2}}\,d\xi + \frac 1{R^{2\ell-2}} \int_{|\xi| > R} {|\hat f_k^\infty(\xi)|^2}\,d\xi \le \\
 & \frac 2{3-2\ell} R^{3-2\ell} + \frac 1{R^{2\ell-2}} \int_{\R} {|\hat f_k^\infty(\xi)|^2}\,d\xi  = 
  \frac 2{3-2\ell} R^{3-2\ell} +
\frac{2\pi}{R^{2\ell-2}} \int_{\R_+} {|f_k^\infty(x)|^2}\,dx, 
\end{equations}
thanks to the Plancherel equivalence. Hence, optimizing over $R$ in \fer{eq:cs} we obtain
\be\label{eq:stima12}
\int_\R \frac{|\hat f_i^\infty(\xi)|^2}{|\xi|^{2\ell-2}}\,d\xi \le C_\ell \| f_k^\infty\|_{L^2(\R_+)}^{3/2-\ell},
\ee
where the constant $C_\ell$ is given by
\be\label{eq:con12}
C_\ell = (2\pi)^{3-2\ell} \left[ 3 \left(\frac{2\ell-2}3 \right)^{3-2\ell} + \left(\frac 3 {2\ell-2}\right)^{2\ell-2}\right], 
\ee
which is positive for any $\ell>1$. Hence, we get
\be\label{eq:sti12}
 \int_\R \frac{|\hat f_k(\xi) -\hat f_k^\infty(\xi)|}{|\xi|^{2\ell}}\,  |\xi \hat f_k^\infty(\xi)|\,d\xi  \le \left(C_\ell \| f^\infty_k\|_{L^2(\R_+)}^{3/2-\ell} \right)^{1/2} \left( \int_\R \frac{|\hat f_k(\xi) -\hat f_k^\infty(\xi)|^2}{|\xi|^{2\ell}}\,d\xi\right)^{1/2}.
\ee
Analogous result can be obtained for the integral
\[
 \int_\R \frac{|\hat f_k(\xi) -\hat f_k^\infty(\xi)|}{|\xi|^{2\ell}}\,\left| \xi \frac{\partial\hat f_k^\infty(\xi)}{\partial\xi}\right|\,d\xi .
\]
In this case, from the Plancherel equivalence we get
\[
\int_\R \left|\frac{\partial\hat f_k^\infty(\xi)}{\partial\xi}\right|^2\,d\xi = 2\pi \int_{\R_+} x^2| f_k^\infty(x)|^2\, dx,
\]
which is bounded for any $p \in [1/2,1]$ from \eqref{eq:equi-Gamma},  we conclude with the following inequality
\be\label{eq:sti1}
 \int_\R \frac{|\hat f_k(\xi) -\hat f_k^\infty(\xi)|}{|\xi|^{2\ell}}\,  |\xi \hat f_k^\infty(\xi)|\,d\xi  \le \left(C_\ell \| xf_k^\infty\|_{L^2(\R_+)}^{3/2-\ell} \right)^{1/2} \left( \int_\R \frac{|\hat f_k(\xi) -\hat f_k^\infty(\xi)|^2}{|\xi|^{2\ell}}\,d\xi\right)^{1/2}.
\ee
Analogous results are obtained if $p = 1$, for which we get
\begin{equations}\label{eq:p11}
&\frac d{dt} \int_\R \frac{|\hat f_k(\xi) -\hat f_k^\infty(\xi)|^2}{|\xi|^{2\ell}}\, d\xi \\
&\quad\le -\left[ \sigma^2_k(t) \dfrac{3-2\ell}{4} + {\lambda_k}(t)\right](2\ell-1) \int_\R \frac{|\hat f_k(\xi) -\hat f_k^\infty(\xi)|^2}{|\xi|^{2\ell}}\, d\xi +\\
&\qquad  2{A_k}(t)  \int_\R \frac{|\hat f_k(\xi) -\hat f_k^\infty(\xi)|}{|\xi|^{2\ell}}\,\left| \xi \frac{\partial\hat f_k^\infty(\xi)}{\partial\xi}\right|\,d\xi +\\
&\qquad2 {B_k}(t)  \int_\R \frac{|\hat f_k(\xi) -\hat f_k^\infty(\xi)|}{|\xi|^{2\ell}}\,  \left|\xi \hat f_k^\infty(\xi)\right|\,d\xi.
\end{equations}

Finally, we have proven the result
\begin{theorem}\label{th:p-12}
Let $f_k(x,t)$, $k=1,2$ be the solutions of the Fokker--Planck equations \fer{eq:FP-gen}, with initial values the probability densities $f_{k,0}$  such that, for $p= \frac 1 2, 1$ and $1< \ell<\frac 3 2$
\[
\|f_{k,0} - f_k^\infty\|_{\dot{H}_{-\ell}} \ < +\infty, \qquad  k = 1,2, 
\]
where $f_k^\infty$ are the Gamma equilibrium solutions  \fer{eq:equilibrium_G} or  inverse Gamma equilibrium solutions \eqref{eq:equilibrium_invG}.  Then, for $p=1/2, 1$ the solutions to \fer{eq:FP-gen}, and consequently of system \eqref{intro:systFP} satisfy
\[
 \|f_{k}(t) - f_k^\infty\|_{\dot{H}_{-\ell}}\to 0 ,\qquad  k = 1,2,
\]
 The rate in the case $p =1/2$ is determined by inequality \fer{eq:sti12}, while in the case $p=1$ is determined by inequality \fer{eq:sti1}.
\end{theorem}

{
\begin{remark}
We observe that if $p=1/2$ the functions $f_k^\infty(x)$ and $x f_k^\infty(x)$ are always in $L^2(\mathbb R_+)$ since all the moments are finite. In the case $p=1$ the quantities $f_k^\infty(x)$ and $x f_k^\infty(x)$ are in $L^2(\mathbb R_+)$ under the hypotesis of bounded variance, holding under the assumptions $\bar\nu^\infty_k>3$, see \eqref{eq:nu_p1}.
\end{remark}
}

\subsection{Convergence to quasi-equilibrium densities}\label{sec:quasi-lim}

The results of the previous sections allow  to study the role played by the quasi-equilibrium solutions \fer{eq:gen-Gamma} defined in Section \ref{sect:41}. Indeed, as far as Cramér distance is concerned, triangular inequality gives
\be\label{eq:tria}
d(f_k(t), f_k^q(t)) ^{1/2} \le d(f_k(t), f_k^\infty) ^{1/2}+ d(f_k^\infty, f_k^q(t)) ^{1/2}.
\ee
Now, while the decay of the solutions $f_k(x,t)$ towards equilibrium has been obtained in the previous Sections, the decay of the quasi-equilibria $f_k^q/x,t)$ towards equilibrium can be easily found by resorting to the notion of relative entropy between probability densities $f$ and $g$ , expressed by
\[
H(f|g) = \int_{\R_+} f(x) \log \frac{f(x)}{g(x)} \, dx.
\]
In the present case, for $k=1,2$, we get from \eqref{eq:gen-Gamma}-\eqref{eq:equi-Gamma} 
\[
H(f_k^\infty| f_k^q(t)) = \int_{\R_+} f_k^\infty(x) \left[ \log \dfrac{C_k^{p,\infty}}{C_k^p} - \dfrac{x^{2-2p}}{1-p} \left(\dfrac{\lambda_k}{\sigma^2_k} - \dfrac{\lambda_{k,\infty}}{\sigma^2_{k,\infty}} \right) - \dfrac{2x^{1-2p}}{2p-1} \left(\dfrac{\mu_k}{\sigma_k^2}-\dfrac{\mu_{k,\infty}}{\sigma^2_{k,\infty}} \right) \right]dx. 
\]
Since $\lambda_k \to \lambda_{k,\infty}$, $\mu_k \to \mu_{k,\infty}$, $\sigma^2_k \to \sigma^2_{k,\infty}$ as $t\to +\infty$ and since $\int_{\mathbb R_+}f_k^\infty(x) x^{\theta}dx<+\infty$ for any $\theta \in [-1,1]$ we get that 
\[
H(f_k^\infty| f_k^q(t))  \to 0, \qquad t\to+\infty. 
\]
Therefore, thanks to the Czisar-Kullback inequality 
\[
\| f^\infty_k-f^q_k\|_{L^1} \to 0^+, 
\]
which controls the $\dot H_{-\ell}$ norm. Hence, $f_k \to f_k^q$ for any $1\le\ell<\frac 3 2$.

\section{Numerical tests}\label{sect:5}
In this section, we provide several numerical tests to illustrate the obtained decay towards equilibrium for the system of Fokker-Planck equations \eqref{intro:systFP} in terms of the energy distance \eqref{eq:norm}. In particular, we illustrate the consistency of the decays through the class of structure preserving scheme introduced in \cite{PZ} to approximate the evolution of the distribution functions $f_1,f_2$. These methods are capable to reproduce large times statistical properties of the exact steady state with arbitrary accuracy, together with the preservation of positivity of the solution and a consistent entropy dissipation. In all the subsequent tests we initialise the parameters defined in Table \ref{Table:params}. {In all the subsequent tests we considered a domain $[0,L]$, $L=50$, discretized by means of $N = 1001$ gridpoints. The discretization of the considered time interval has time step $\Delta t >0$ such that $\Delta t = \Delta x/2$ to comply with positivity preservation os semi-implicit schemes discussed in \cite{PZ}. At the boundary $x =0$ we always considered a no-flux boundary condition }

\begin{table}
  \centering
  \caption{Set of parameters considered.}
  \begin{tabular}{|c|c|c|}
    \hline
    \textbf{Parameter} & \textbf{Value} & \textbf{Meaning}\\
    \hline\hline
    $\alpha$ & 1.0 &  Birth rate $f_1$ \\ 
    $\beta$ & 0.5& Removal rate due to contact between $f_1$ and $f_2$ \\
    $\gamma$ & 0.15 & Birth rate $f_2$ due to contact with $f_1$\\
    $\sigma_1$ & 0.05 & Diffusion strength $f_1$ \\
    $\sigma_2$ & 0.05 & Diffusion strength $f_2$ \\
    $K$ & 0.01 & Carrying capacity  \\
    $\delta$ & 0.5 & Death rate $f_2$ \\
    $\chi$ & 0 & Intensity redistribution particles $f_1$ \\
    $\theta$ & 0 & Intensity redistribution particles $f_2$\\
    $\nu$ & 1.0 &  Birth rate $f_1$ \\
     $\mu$ & 10 & Threshold death rate $f_2$ \\
     $\delta = \gamma\mu-\nu$ & $0.5$ & Death rate $f_2$\\
    \hline\hline
  \end{tabular}
  \label{Table:params}
\end{table}

\subsection{Test 1: convergence to the quasi-equilibrium densities}
In this section we show the evolution of the energy distance $\mathcal E_\ell(f_k,f_k^q)$, $k=1,2$, defined in \eqref{eq:norm} as discussed in Section \ref{sect:2}. In more detail, we will consider the initial distributions
\[
\begin{split}
f_1(x,0) = 
\begin{cases}
1 & x \in [m_1(0) - \frac 1 2, m_1(0) + \frac 1 2], \\
0 & \textrm{elsewhere}
\end{cases}
\\
f_2(x,0) = 
\begin{cases}
1 & x \in [m_2(0) - \frac 1 2, m_2(0) + \frac 1 2], \\
0 & \textrm{elsewhere}
\end{cases}
\end{split}
\]
where $m_1(0) = 4$, $m_2(0) = 3$. In Figure \ref{fig:quasi_p12}-\ref{fig:quasi_p1} we show the dynamics of the Energy distance for several choices of $\ell = 0.6,0.7,0.8$. The evolution of the densities $f_1,f_2$ have been approximated through a second order semi-implicit structure preserving scheme over the domain $[0,L]$, $L = 50$ discretized with $N = 1001$ gridpoints, we point the interested reader to \cite{PZ} for a detailed description. We integrated the dynamics over the time horizon $[0,50]$ with a time step $\Delta t = \Delta x/2$. The damping of the Energy distance can be observed, with the upper bounds derived in Section \ref{sect:4} indicating that the convergence becomes exponential after a transient regime.
 
 In FIgure \ref{fig:f_fq} we depict the approximation of $f_1,f_2$ obtained with the SP scheme as solution to \eqref{intro:systFP} and the coefficients \eqref{eq:values} and of the quasi-equilibria $f_1^q,f_2^q$ at three times $t_1 = 1$, $t_2 = 10$, $t_3 = 20$. In the top row we consider $p = 1/2$ and in the bottom row $p = 1$, the quasi-equilibria are therefore defined by \eqref{eq:quasi_eq_G} if $p = 1/2$ and by \eqref{eq:quasi_eq_invG} if $p=1$. We can observe how for large times the two distributions converge towards the same limit. 
 
\begin{figure}
\includegraphics[scale = 0.25]{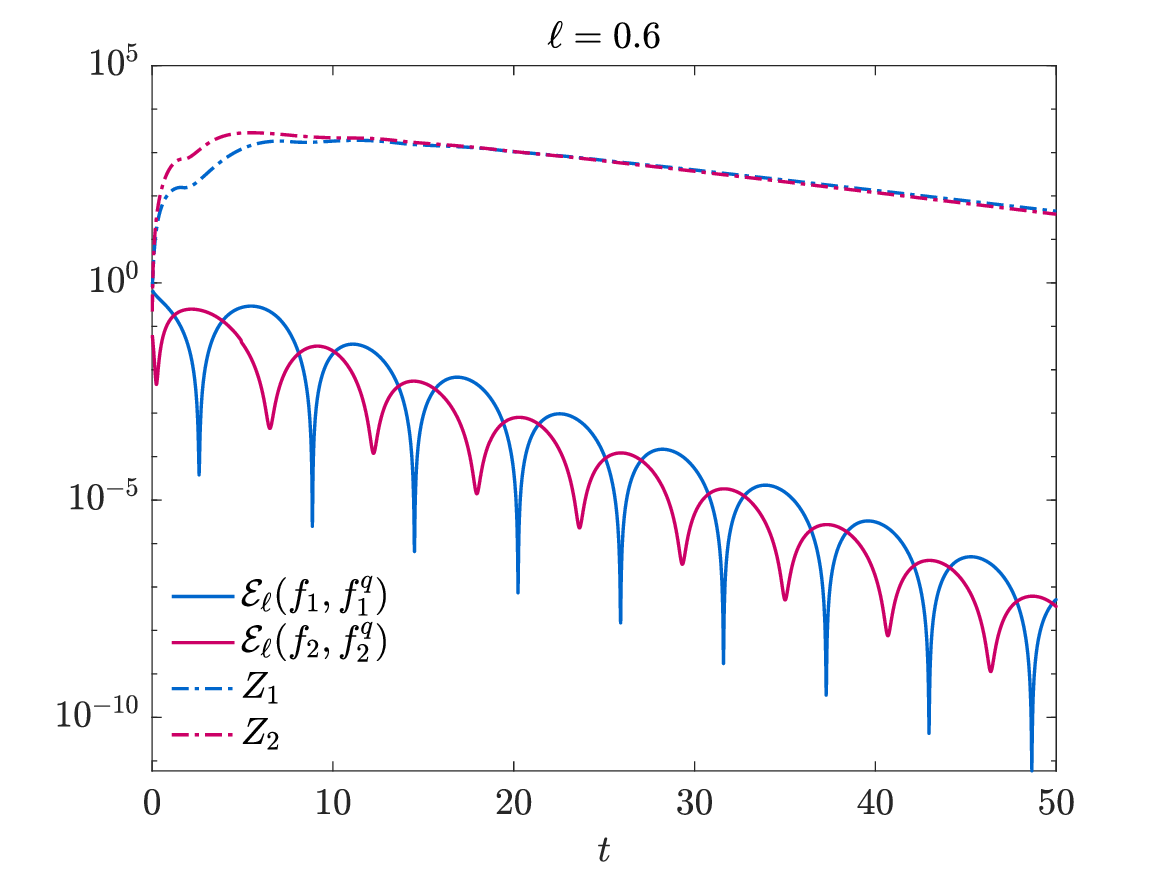}
\includegraphics[scale = 0.25]{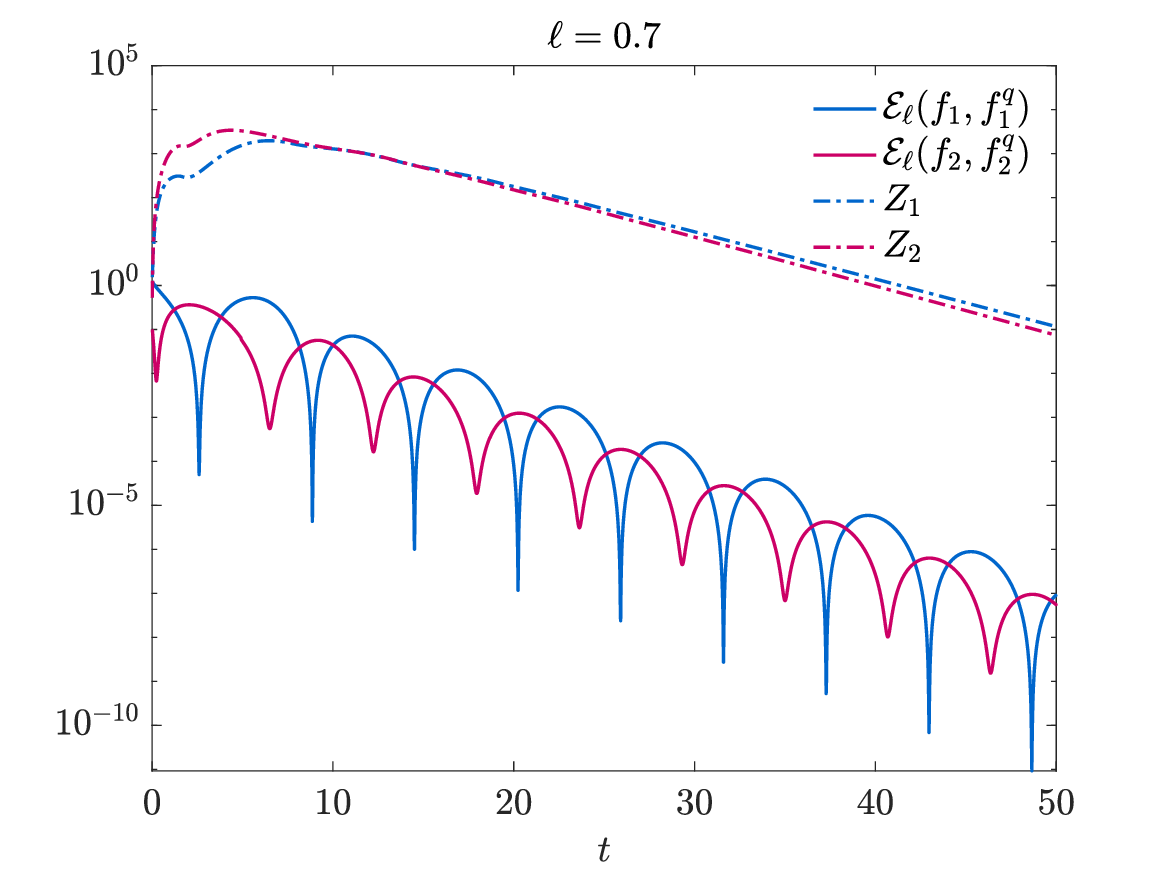}
\includegraphics[scale = 0.25]{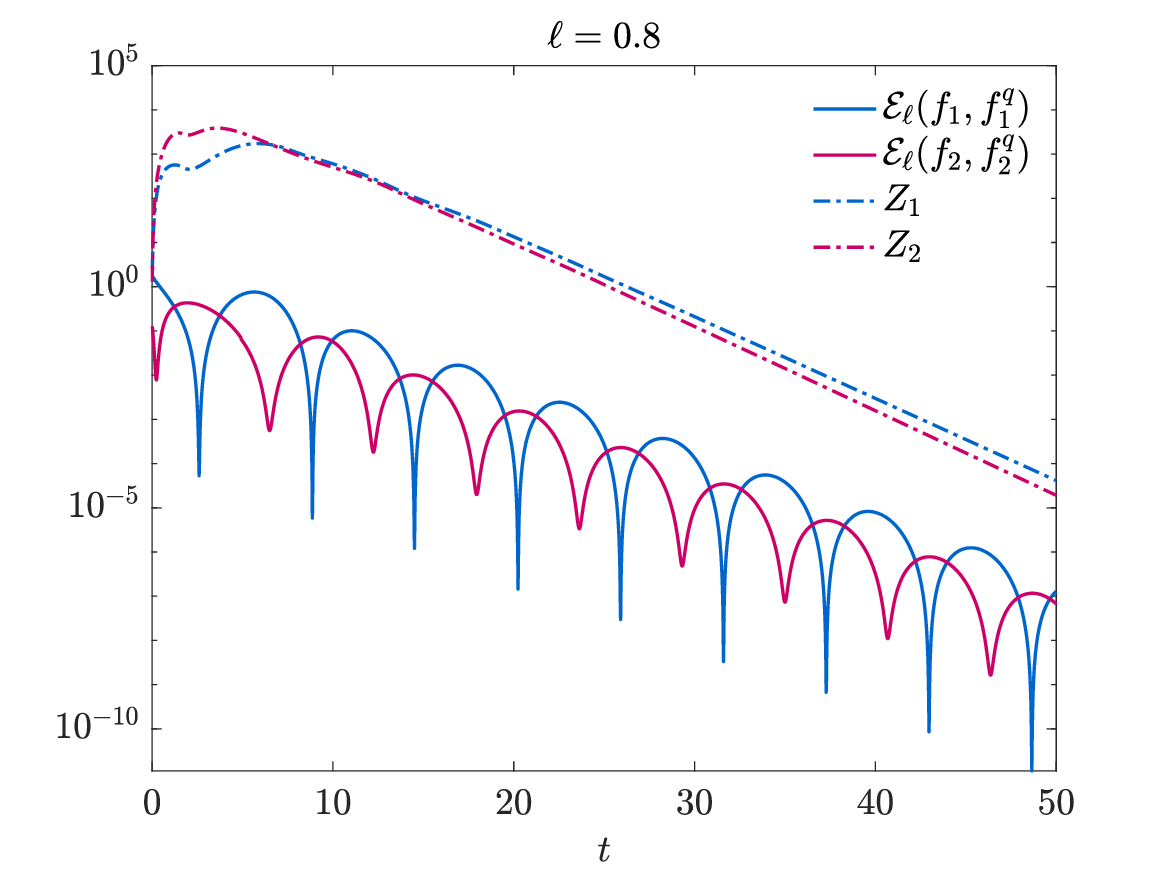}
\caption{Evolution of the Energy distance  $\mathcal E_\ell(f_1,f_1^q)$ and $\mathcal E_\ell(f_2,f_2^q)$ from the solution to  \eqref{intro:systFP} with $p = 1/2$ and the coefficients \eqref{eq:values}. The upper bounds are the ones obtained in Section \ref{sect:4}.  }
\label{fig:quasi_p12}
\end{figure}

\begin{figure}
\includegraphics[scale = 0.25]{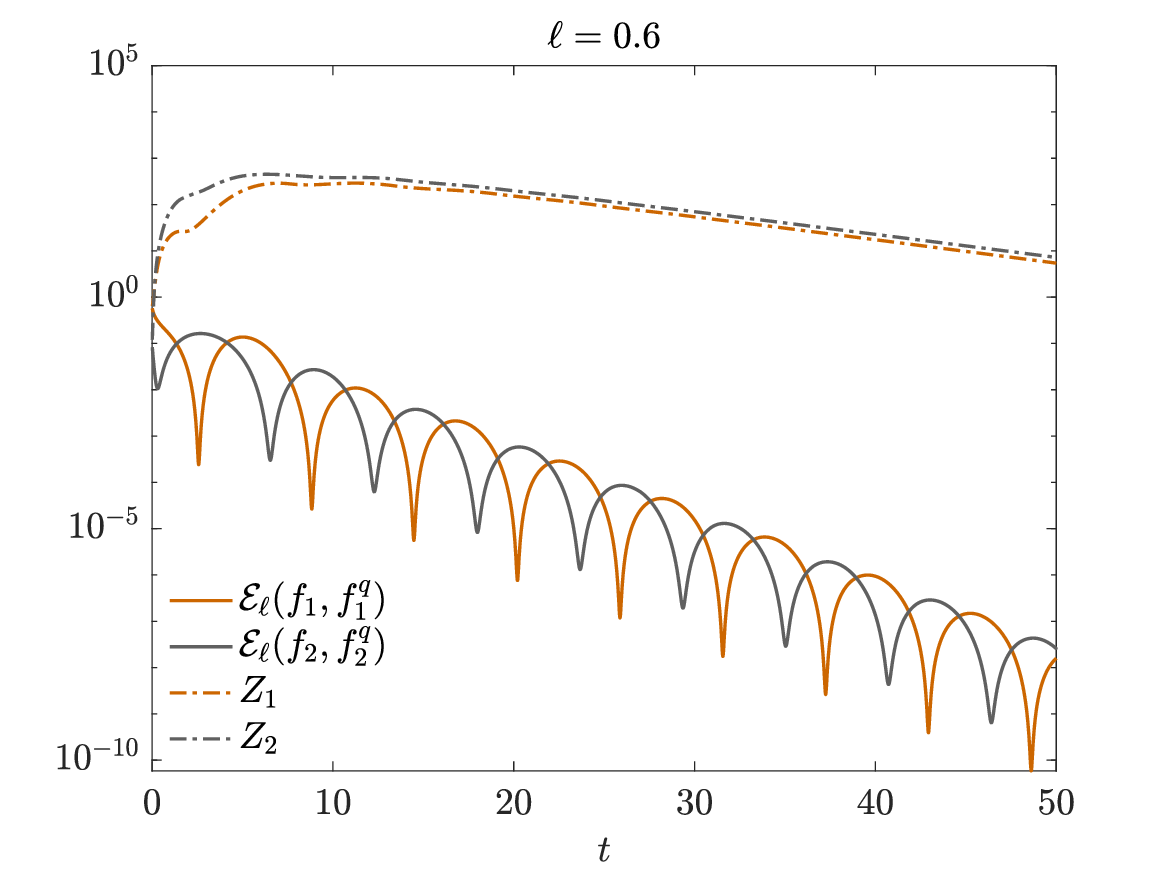}
\includegraphics[scale = 0.25]{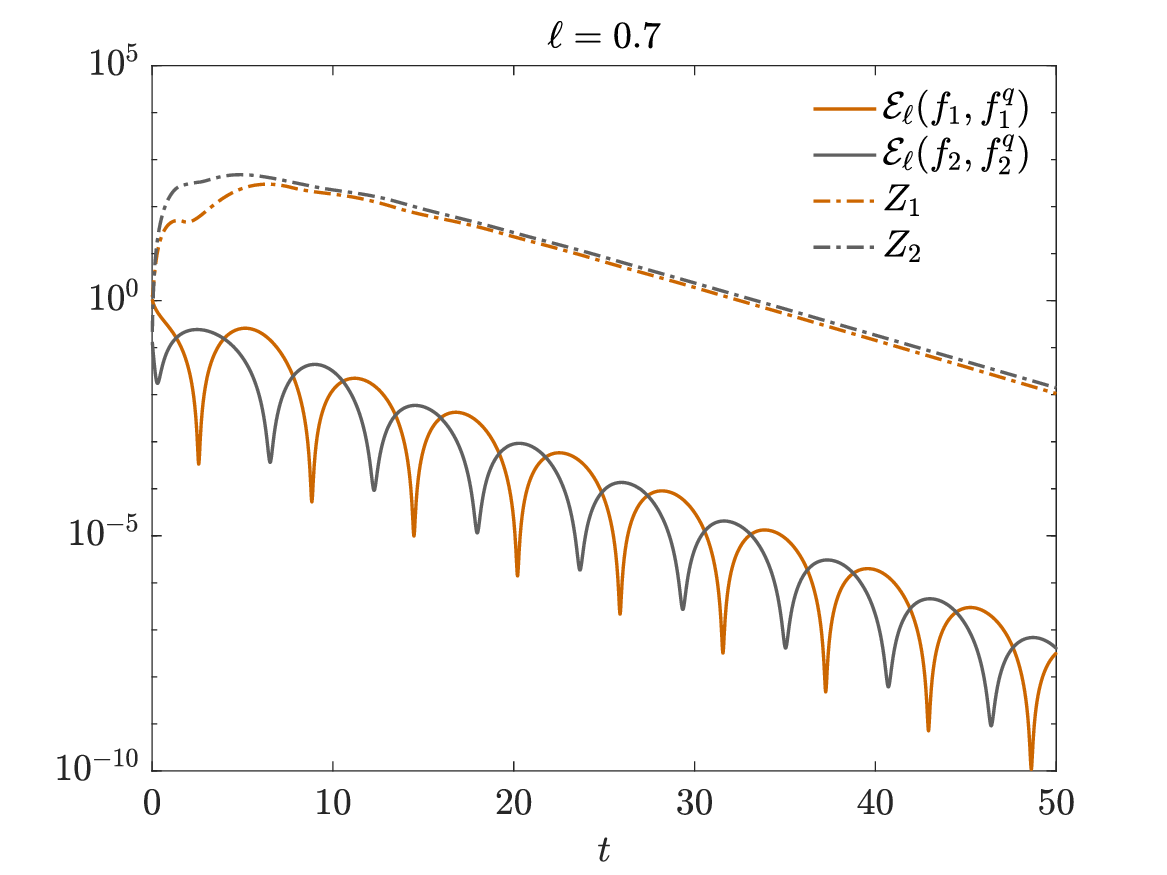}
\includegraphics[scale = 0.25]{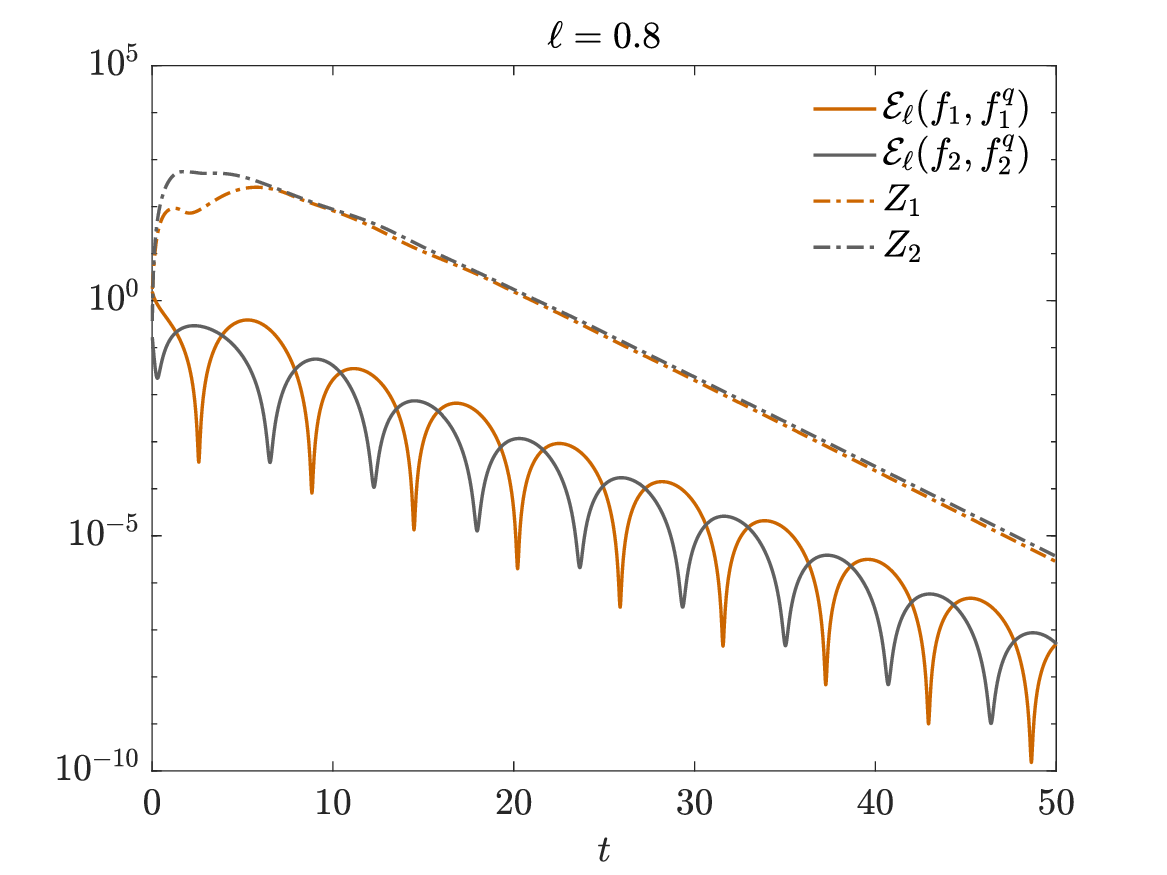}
\caption{Evolution of the Energy distance  $\mathcal E_\ell(f_1,f_1^q)$ and $\mathcal E_\ell(f_2,f_2^q)$ from the solution to  \eqref{intro:systFP} with $p = 1$ and the coefficients \eqref{eq:values}. The upper bounds are the ones obtained in Section \ref{sect:4}.}
\label{fig:quasi_p1}
\end{figure}

\begin{figure}
\centering
\includegraphics[scale = 0.35]{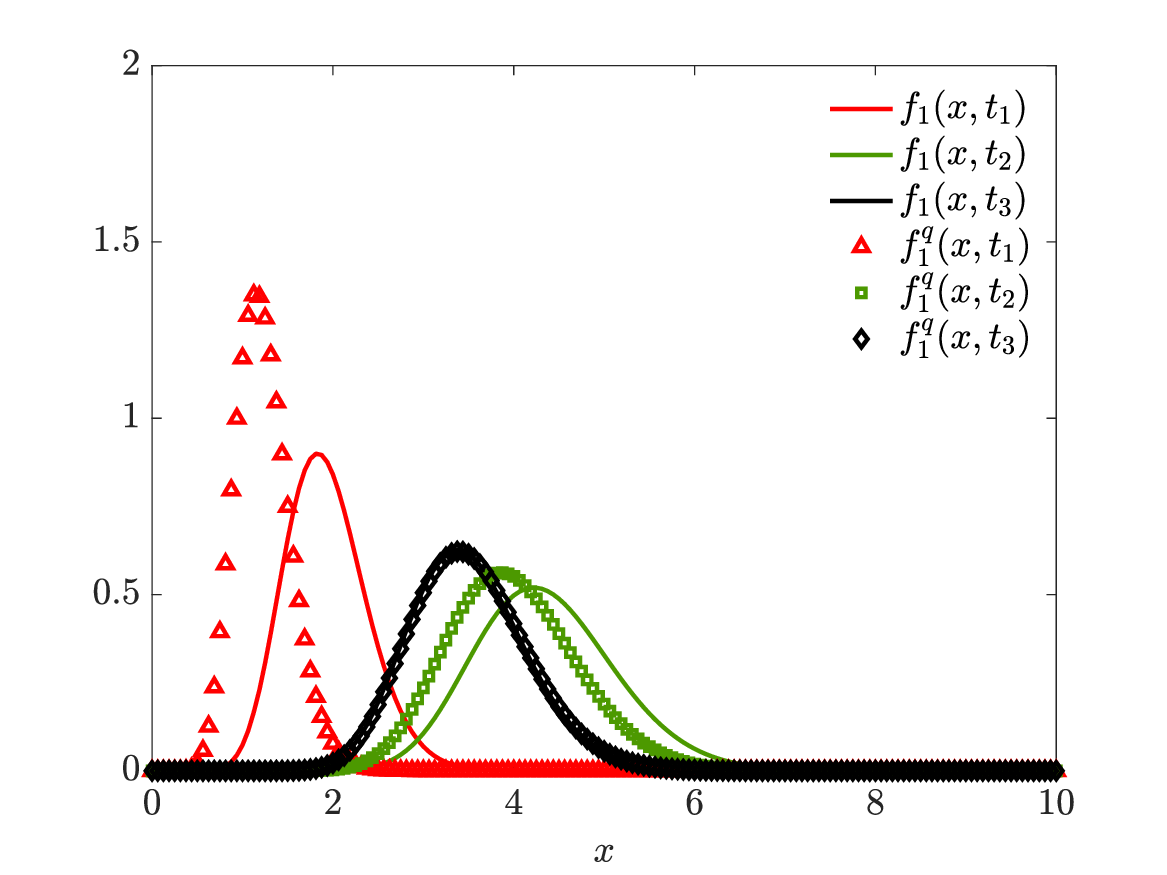}
\includegraphics[scale = 0.35]{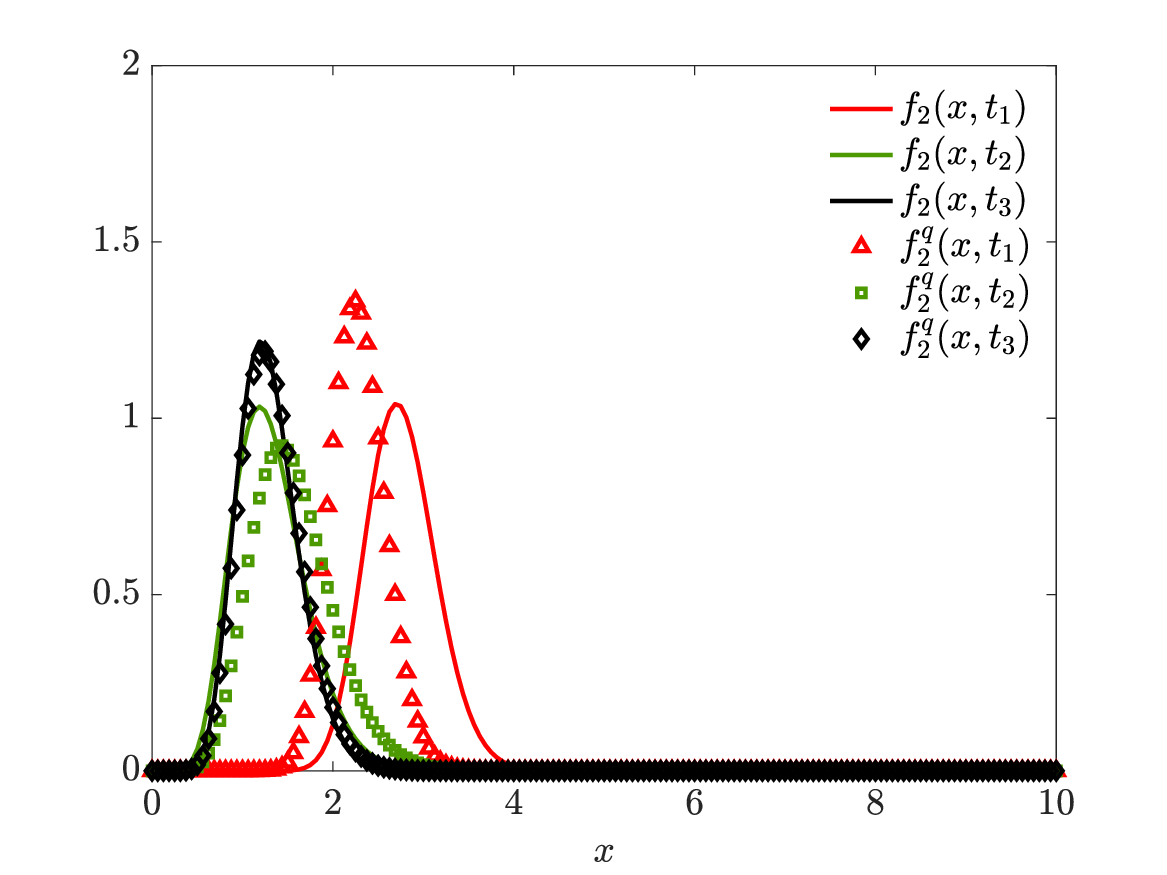}\\
\includegraphics[scale = 0.35]{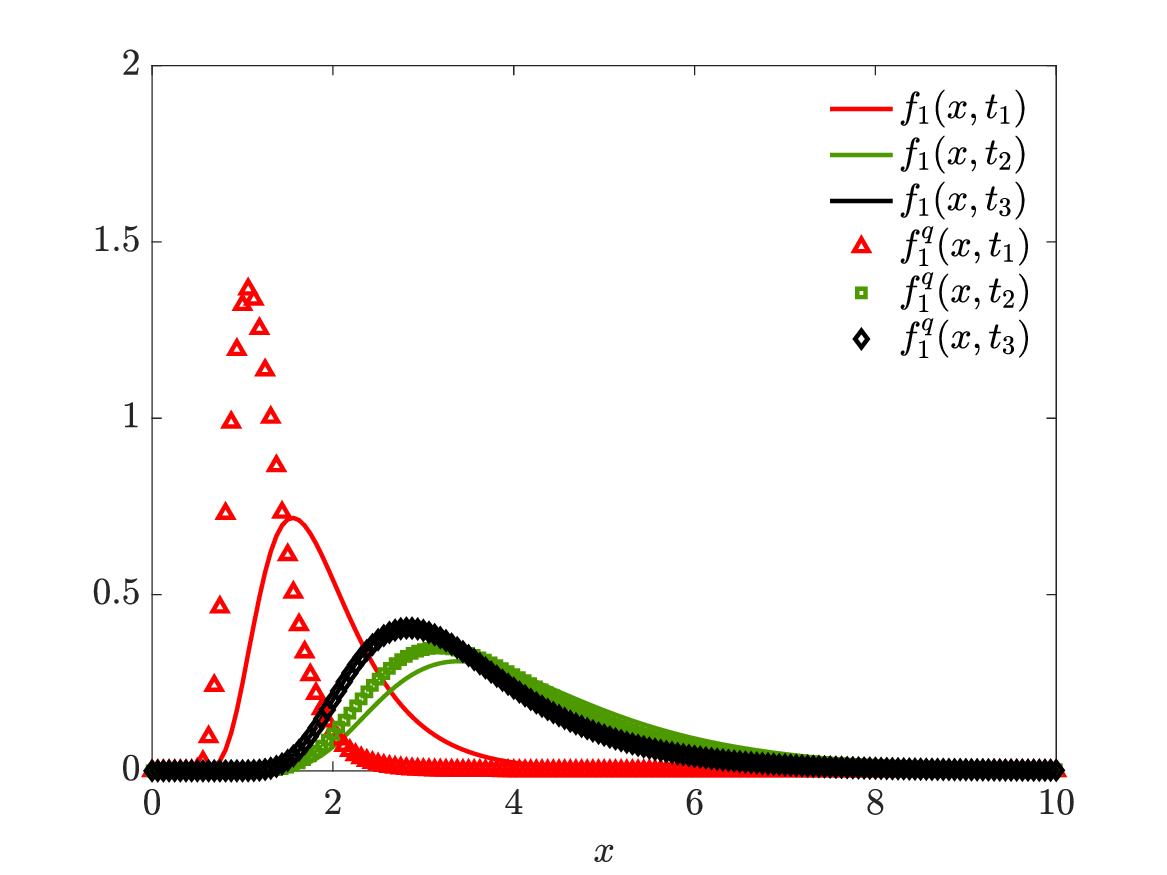}
\includegraphics[scale = 0.35]{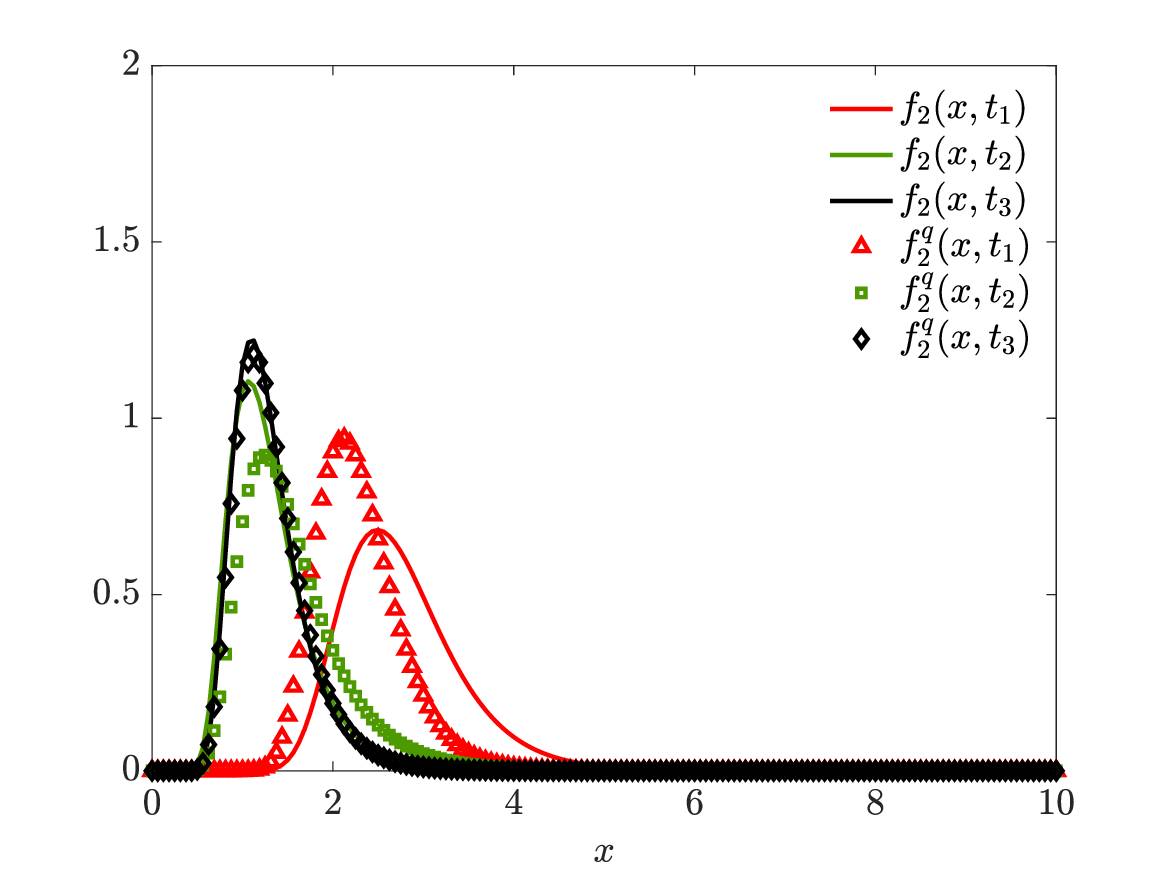}
\caption{We represent $f_1,f_2$ and $f_1^q,f_2^q$ for three time steps $t_1 = 1$, $t_2 = 10$, $t_3 = 20$. The approximation of $f_1,f_2$ is obtained with a 2nd order  semi-implicit SP scheme for the Fokker-Planck system of equations \eqref{intro:systFP} under the choice of parameters in \eqref{eq:values} and coefficients in Table \ref{Table:params}.    }
\label{fig:f_fq}
\end{figure}

\subsection{Test 2: convergence to the equilibrium density for $p \in [\frac 1 2 ,1]$}
In this section we concentrate the convergence to equilibrium obtained in Section \ref{sec:equi} for several choices of the diffusion strength $x^{2p}$ and $p \in [\frac 1 2, 1]$ for which we obtained asymptotic convergence to the stationary profiles $f_1^\infty,f_2^\infty$ in the Cramér distance $\mathcal E_1$. We consider the evolution of $\mathcal E_1(f_k,f_k^\infty)$, $k=1,2$, for the two interacting densities $f_1,f_2$. In view of the $L^2$ regularity of the equilibrium densities, from the obtained bound, we have an asymptotic exponential convergence to equilibrium provided by the lower bounds on the coefficients $\lambda_k(t)$, $k = 1,2$. 

In Figure \ref{fig:equi} we represent the obtained trends. The approximation of $f_1,f_2$ have been obtained through a 6th order semi-implicit SP scheme for the Fokker-Planck system \eqref{intro:systFP} over the domain $[0,50]$ discretized with $1001$ gridpoints and a time step $\Delta t = \Delta x/2$. {We observe that the computed rate is suboptimal and apparently the system of equations converge faster than the obtained exponential rate. }

\begin{figure}
\centering
\includegraphics[scale = 0.35]{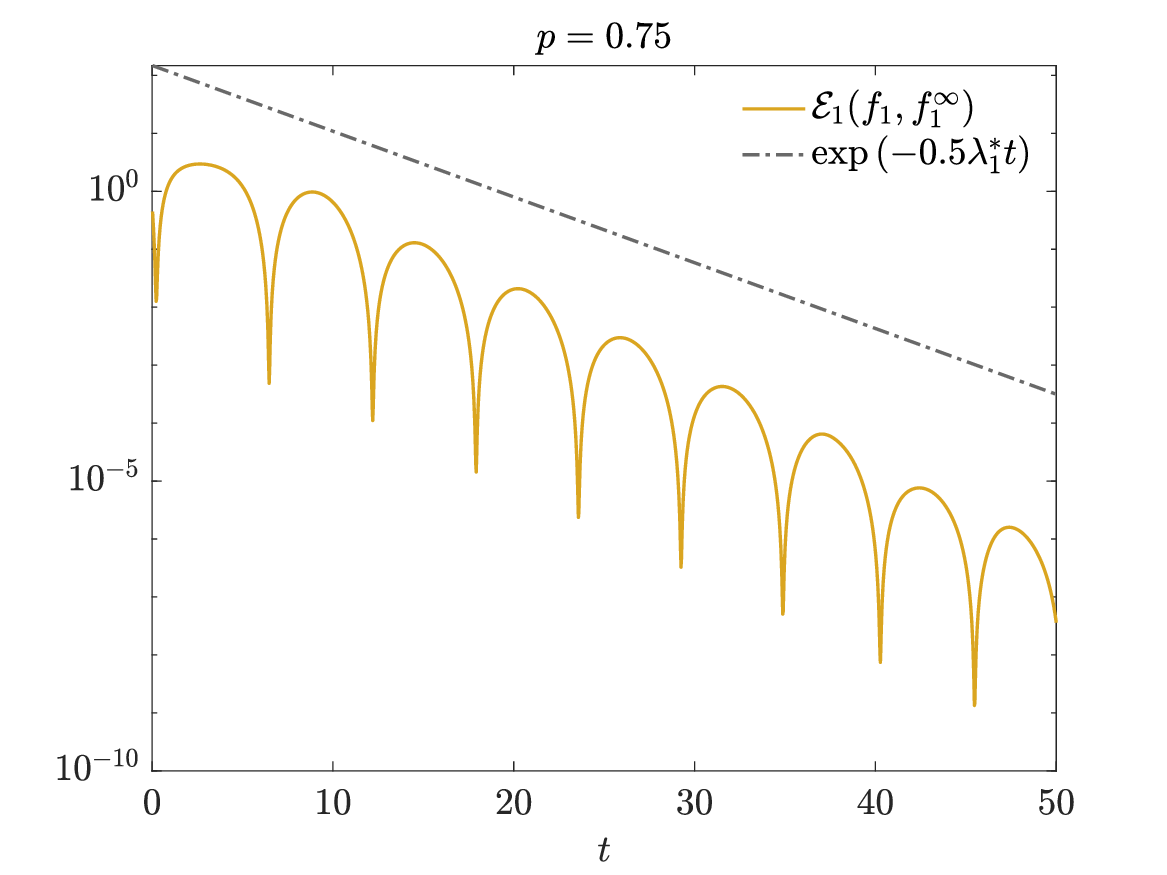}
\includegraphics[scale = 0.35]{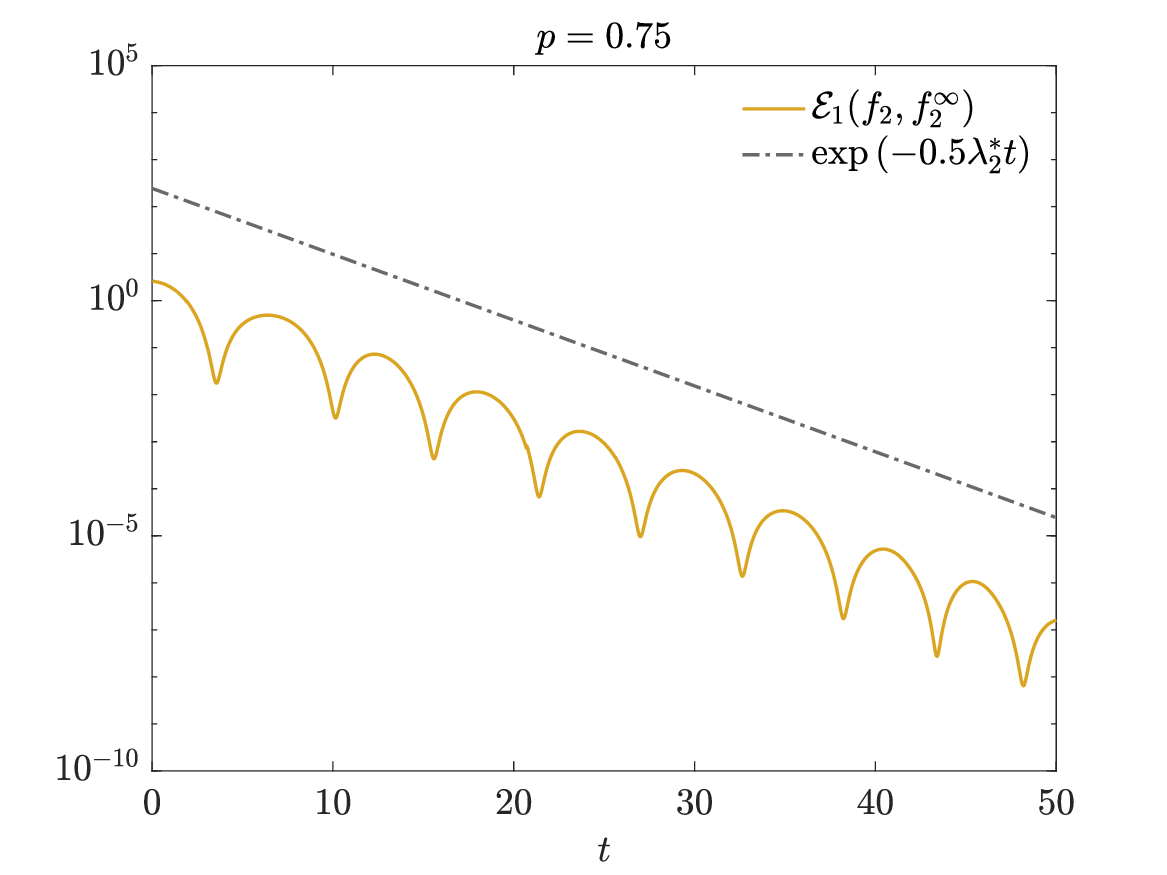}
\caption{Evolution of $\mathcal E_1(f_k,f_k^\infty)$, $k=1,2$, and of the obtained exponential trend to equilibrium.}
\label{fig:equi}
\end{figure}

\section*{Conclusion}
Our analysis establishes a rigorous framework connecting kinetic descriptions of interacting populations with their macroscopic Lotka-Volterra counterparts. Through the introduction of Energy-type distances, we demonstrated the dissipative structure of the considered dynamics, which drives exponentially in time the system toward equilibrium. Thanks to the obtained bounds we quantified the rate of convergence in suitable homogeneous Sobolev spaces and we also clarified the underlying energy dissipation mechanism governing the dynamics. Overall, the proposed approach bridges kinetic modelling with Lotka-Volterra-type dynamics establishing a new perspectives in the derivation of equilibration rates in large interacting systems. {More generally, we expect that the developed techniques can be fruitfully applied to a vast class of modelling settings. }

{
\section*{Data Accessibility}
All the sources and data of the manuscript have been provided in the following repository https://doi.org/10.5281/zenodo.18756863 
}
\section*{Acknowledgements}
M.Z. acknowledges partial support by PRIN2022PNRR project No. P2022Z7ZAJ, European Union - NextGenerationEU and by ICSC – Centro Nazionale di Ricerca in High Performance Computing, Big Data and Quantum Computing, funded by European Union – NextGenerationEU.

\end{document}